\documentclass[11pt]{article}

\makeatletter
\usepackage[T1]{fontenc}
\usepackage{amsmath,amssymb,xcolor,url,mathtools,amsthm}
\usepackage{graphicx}
\usepackage[normalem]{ulem}
\usepackage{algorithm}
\usepackage{algpseudocode}
\usepackage{geometry}
\usepackage{authblk}
\DeclareRobustCommand{\citet}[2][]{%
    \ifthenelse{\equal{#1}{}}%
        {\cite{#2}}%
        {\cite[#1]{#2}}%
}
\DeclareRobustCommand{\citep}[2][]{%
    \ifthenelse{\equal{#1}{}}%
        {\cite{#2}}%
        {\cite[#1]{#2}}%
}

\makeatother

\usepackage{verbatim}
\usepackage{bbm}
\usepackage{cancel}
\usepackage{subcaption}

\input{notation}

\begin{document}

\makeatletter
\title{{\emph{A Priori}} Error Bounds for the Approximate Deconvolution Leray Reduced Order Model}

\author[1]{Ian Moore}
\author[2]{Anna Sanfilippo}
\author[3]{Francesco Ballarin}
\author[1]{Traian Iliescu}
\affil[1]{Department of Mathematics, Virginia Tech, 225 Stanger Street, Blacksburg, Virginia 24061, US}
\affil[2]{Department of Mathematics, Universit{\`a} di Trento, via Sommarive 14, Povo, 38123, Italy}
\affil[3]{Department of Mathematics and Physics, Universit{\`a} Cattolica del Sacro Cuore, via Garzetta 48, Brescia, 25133, Italy}

\date{}

\makeatother

\def\abstractcontent{
The approximate deconvolution Leray reduced order model (ADL-ROM) uses spatial filtering to increase the ROM stability, and approximate deconvolution to increase the ROM accuracy.
In the under-resolved numerical simulation of convection-dominated flows, ADL-ROM was shown to be significantly more stable than the standard ROM, and more accurate than the Leray ROM.
In this paper, we prove 
{\emph{a priori}} error bounds for the approximate deconvolution operator and 
ADL-ROM.
To our knowledge, these are the first numerical analysis results for approximate deconvolution in a ROM context.
We illustrate these numerical analysis results in the numerical simulation of convection-dominated flows.
}

\def\keywordscontent{Reduced order model; POD-Galerkin; convection-dominated flows; spatial filter; approximate deconvolution; Leray model.}

\makeatletter
\maketitle
\begin{abstract}
\abstractcontent
\end{abstract}

\makeatother

\section{Introduction}

Reduced order models (ROMs) are surrogate models that can reduce the dimension of full order models (FOMs) (i.e., classical discretizations with, e.g., finite element method) by orders of magnitude.
Galerkin ROMs (G-ROMs) are ROMs that employ a data-driven basis in a Galerkin framework:
First, data is used to construct a basis $\{ \bvarphi_1, \ldots, \bvarphi_R \}$ that is optimal with respect to the training data (i.e., problem and parameters).
Then, the first $r$ basis functions (which represent the dominant structures in the training data) are used in a Galerkin formulation of the given PDE.
This yields an $r$-dimensional ODE system, where $r$ can be orders of magnitude lower than the FOM dimension.

G-ROMs have been successful in the numerical simulation of diffusion-dominated (laminar) flows, at low Reynolds numbers.
However, the standard G-ROMs generally fail in the numerical simulation of convection-dominated (e.g., turbulent) flows, at high Reynolds numbers.
We emphasize that this is a major drawback of standard G-ROMs since the vast majority of flows in engineering (e.g., aerospace and automotive industry) and 
geophysics (e.g., 
weather prediction and climate modeling) 
are convection-dominated.

The main reason for the poor G-ROM performance is that  convection-dominated (e.g., turbulent) flows exhibit a wide range of spatial scales.
To approximate all the spatial scales requires a large number of basis functions, which yields a 
high-dimensional G-ROM.
Since G-ROM needs to be computationally efficient (and, thus, relatively low-dimensional), practitioners  truncate the ROM basis.
We emphasize, however, that this severe truncation yields inaccurate G-ROM results, usually in the form of numerical oscillations.
The reason is that, although the truncated ROM basis functions $\{ \bvarphi_{r+1}, \ldots, \bvarphi_R \}$ represent the small spatial scales, these scales play an important role in the flow physics.
Indeed, as shown in~\citet{CSB03}, the role of the higher index ROM basis function is to dissipate energy from the ROM system, just as in the classical Fourier setting~\citep{Pop00}.
Since this physical mechanism is not included in the standard G-ROMs, the energy that should be dissipated is incorrectly represented by the relatively low-dimensional basis $\{ \bvarphi_1, \ldots, \bvarphi_r \}$ and generates unphysical oscillations.

There are essentially two strategies to tackle these unphysical oscillations yielded by under-resolved G-ROMs in convection-dominated flows: (i) ROM closures, and (ii) ROM stabilizations.
ROM closures, which are reviewed in~\citet{ahmed2021closures}, add an extra term (a closure model) to the standard G-ROM in order to represent the effect of the discarded modes $\{ \bvarphi_{r+1}, \ldots, \bvarphi_R \}$.
ROM stabilizations (see~\citet{parish2023residual} for a review) either modify the standard G-ROM or add extra terms to it with the specific goal of increasing the G-ROM's numerical stability.
Regularized ROMs (Reg-ROMs) are stabilizations strategies that use spatial filtering to increase the G-ROM numerical stability.
We note that the ROM closures and stabilizations often share similar features since they both aim at increasing the G-ROMs stability (albeit using different strategies).

Numerical stabilizations and closures are nowadays accompanied with extensive theoretical support and numerical analysis when employed with classical numerical discretizations, such as the finite element method.
For instance, the mathematical analysis of several classes of large eddy simulation (LES) models is discussed in the monographs~\citet{BIL05,rebollo2014mathematical,john2004large}, alongside the numerical analysis of their discretization. Furthermore, the mathematical and numerical analysis of stabilization strategies is presented in the monograph~\citet{roos2008robust}.
The aforementioned references tackle on the one hand fundamental numerical analysis questions (such as stability and convergence of the discrete methods), and on the other hand practical challenges (in particular, parameter scalings for stabilization coefficients), which are critical in the application of stabilization techniques in real applications. 

In contrast, despite the recent interest in ROM closures~\citep{ahmed2021closures} and stabilizations~\citep{parish2023residual}, providing mathematical and numerical analysis support for them is generally an open problem.
Indeed, the ROM closures and stabilizations are generally assessed heuristically:
The proposed model is used in numerical simulations and is shown to improve the numerical accuracy of the standard G-ROM and/or other ROM closure models.
However, fundamental questions in ROM numerical analysis (e.g., stability, convergence, and rate of convergence) are still wide open for most of the ROM closures and stabilizations.
Only the first steps in the numerical  analysis  of ROM closures~\citep[Section 4.6]{ballarin2024apriori} and stabilizations~\citep[Section 4.5]{ballarin2024apriori} have been taken.
Despite these first steps, the numerical analysis of ROM closures and stabilizations is nowhere close to the numerical analysis for FOM closures and stabilizations.
In this paper, we 
take a step in this direction and prove {\emph{a priori}} error bounds for the van Cittert approximate deconvolution operator and 
a recent Reg-ROM, the approximate deconvolution Leray ROM (ADL-ROM), which was proposed in~\citet{sanfilippo2023approximate}.
To our knowledge, these are the first numerical analysis results for approximate deconvolution at a ROM level.

The rest of the paper is organized as follows:
In Section~\ref{sec:fe-preliminaries}, we present notation and preliminaries that will be used throughout the paper.
In Section~\ref{sec:rom-preliminaries}, we include several numerical analysis results for the G-ROM, spatial filter, and approximate deconvolution.
These results are leveraged
in Section~\ref{sec:adl-rom-numerical-analysis} 
to prove 
{\emph{a priori}} error bounds for the ADL-ROM. 
In Section~\ref{sec:numerical-results}, we present numerical results that illustrate the 
{\emph{a priori}} error bounds proved in Sections~\ref{sec:rom-preliminaries} and~\ref{sec:adl-rom-numerical-analysis}.
Finally, in Section~\ref{sec:conclusions}, we draw conclusions and outline future research directions.

\section{Finite Element Preliminaries}
    \label{sec:fe-preliminaries}

In this section, we present notation and standard numerical analysis results for the finite element (FE) discretization of the incompressible Navier-Stokes equations (NSE):
\begin{eqnarray}
\frac{\partial \bu}{\partial t} - \nu \Delta\bu +\bigl(\bu\cdot\nabla\bigr)\bu+\nabla p \,&=\, \bff, & \text{in } \Omega \times (0, T]\label{eqn:nse-1} \\
\nabla\cdot\bu\,&=\,0, & \text{in } \Omega \times (0, T] \label{eqn:nse-2}.
\end{eqnarray}
Here $\Omega \subset \mathbb{R}^d$ is the spatial domain, $d = 2, 3$ is the spatial dimension, $T$ is the final time, $\bu = [u_1, \hdots, u_d]^\top: \Omega \times [0, T] \to \mathbb{R}^d$ is the NSE velocity, and $p: \Omega \times [0, T] \to \mathbb{R}$ is the NSE pressure. Furthermore, $Re$ is the dimensionless Reynolds number, $\nu$ is its inverse, while $\bff: \Omega \times [0, T] \to \mathbb{R}^d$ is a forcing term.
For simplicity, 
we equip the NSE with homogeneous Dirchlet boundary conditions.

In the following, $\| \cdot \|$ and $(\cdot, \cdot)$ refer to the $L^2(\Omega)$ norm and inner product, respectively.
Upon introducing the spaces
\begin{align*}
    Q &= L_0^2(\Omega) := \left\{ q \in L^2(\Omega) , \int_\Omega q = 0 \right\}, \\
    \bX &= H_0^1(\Omega;\ \mathbb{R}^d) := \left\{ \bu \in H^1(\Omega;\ \mathbb{R}^d), \bu|_{\partial\Omega}  = \bzero \right\},
\end{align*}
we can now define the weak formulation of the NSE (\ref{eqn:nse-1})-(\ref{eqn:nse-2}) as: Find $\bu\in\bX$ and $p\in Q$ such that
\begin{equation}
\begin{cases}
\displaystyle\int_{\Omega} \frac{\partial \bu}{\partial t}\cdot \bv \; d{\bx}
	+ \nu (\nabla \bu,\nabla \bv)
	+ b(\bu, \bu, \bv) - (\nabla \cdot \bv, p)
	= (\bff, \bv), & \forall \bv \in \bX,\\
(\nabla \cdot \bu, q) = 0, & \forall q \in Q,
\end{cases}
\label{eq:NSE_weak}
\end{equation}
where the trilinear form $b: \bX \times \bX \times \bX \to \mathbb{R}$ is defined as follows:
\begin{eqnarray}
    b(\bu,\bv,\bw)
    = ((\bu \cdot \nabla ) \bv,\bw) .
    \label{eqn:definition-b}
\end{eqnarray}

We state the following regularity assumption 
(refer to Definition 29, Proposition 15, and Remark 10 
in \citet{layton2008introduction}): 
\begin{assumption}\label{ass:reg_ex_sol}
    In (\ref{eqn:nse-1})-(\ref{eqn:nse-2}), we assume that $\bff \in L^{2}\left(0, T ; L^{2}(\Omega;\ \mathbb{R}^d)\right)$, $\bu_{0} \in \bX$, $\bu \in L^{2}\left(0, T ; \bX \right)\bigcap L^{\infty}\left(0, T ; L^{2}(\Omega;\ \mathbb{R}^d)\right)$, $\nabla\bu \in L^{\infty}(0,T;L^2(\Omega;\ \mathbb{R}^{d\times d}))$, $\bu_{t} \in L^{2}\left(0, T ; \bX^{*}\right)$, and $p \in$ $L^{\infty}\left(0, T ; Q\right)$.
\end{assumption}

We next proceed to a time discretization by means of the backward differentiation formula of order 2 (BDF2).
Let $k = 1, \hdots, K - 1$ denote the index of the current time step $t_k$. For simplicity, we will assume that $t_{k + 1} - t_k = \Delta t$, i.e., equispaced timesteps, and $K \Delta t = T$. We introduce now two notations with subscript $_k$: the notation $\bu_k$ is reserved for the evaluation of the (strong) solution to \eqref{eqn:nse-1}-\eqref{eqn:nse-2} at time $t_k$, i.e., $\bu_k = \bu(t_k)$, while $\bu^{\Delta t}_k$ is employed to denote the unknown resulting after time discretization. Similar notations are adopted for the strong pressure $p_k = p(t_k)$ and its time discrete approximation $p^{\Delta t}_k$.

After semi-discretization with the BDF2 scheme, \eqref{eq:NSE_weak} becomes

\begin{equation}
\begin{cases}
\displaystyle\int_{\Omega} \frac{3 \bu^{\Delta t}_{k + 1} - 4 \bu^{\Delta t}_{k} + \bu^{\Delta t}_{k-1}}{2 \Delta t}\cdot \bv \; dx
	+ \nu (\nabla \bu^{\Delta t}_{k + 1},\nabla \bv)
	+ b^*(\bu^{\Delta t}_{k + 1}, \bu^{\Delta t}_{k + 1}, \bv) &
	\\
	\qquad\qquad- (\nabla \cdot \bv, p^{\Delta t}_{k + 1})
	= (\bff_{k + 1}, \bv), & \forall \bv \in \bX,\\
(\nabla \cdot \bu^{\Delta t}_{k + 1}, q) = 0, & \forall q \in Q.
\end{cases}
\label{eq:NSE_weak_semi}
\end{equation}
Following, e.g., \citet{layton2008introduction,temam2001navier}, the skew-symmetric trilinear form $b^*$ which appears in \eqref{eq:NSE_weak_semi} is defined as follows:
$b^*: \bX \times \bX \times \bX \to \mathbb{R}$, such that,
\begin{eqnarray}
b^*(\bu,\bv,\bw) = \frac{1}{2} \left[ ((\bu \cdot \nabla ) \bv,\bw) - ((\bu\cdot \nabla) \bw, \bv) \right] .
    \label{eqn:definition-b-skew-symmetric}
\end{eqnarray}
As explained in~\citep[Section 7.1]{layton2008introduction}, we use the skew-symmetric trilinear form $b^*$ in~\eqref{eq:NSE_weak_semi} instead of the trilinear form that we used in~\eqref{eq:NSE_weak} to ensure the stability of the velocity.
We also note that the trilinear forms $b$ and $b^*$ are equivalent in the continuous case (i.e., $b(\bu,\bv,\bw) = b^*(\bu,\bv,\bw), \forall \, \bu \in \bX, \bv \in \bX, \bw \in \bX$, see~\citep[Lemma 18]{layton2008introduction}), but not in the discrete case. Furthermore, we recall the following result from \citep[Lemmas 13 and 18]{layton2008introduction}:
\begin{lemma}
\label{Lemma:skew_Symmetric}
    For all $\bu,\bw,\bv \in \bX$, the skew-symmetric trilinear form $b^*(\cdot,\cdot,\cdot)$ satisfies
    \begin{equation*}
        b^*(\bu,\bv,\bv) = 0, 
    \end{equation*}
    \begin{equation*}
        b^*(\bu,\bv,\bw) \leq C \, \| \nabla \bu\| \, \| \nabla \bv \| \, \| \nabla \bw \|,
    \end{equation*}
    and a sharper bound
    \begin{equation*}
        b^*(\bu,\bv,\bw) \leq C \sqrt{\|\bu\| \, \| \nabla \bu\|}  \, \| \nabla \bv \| \, \| \nabla \bw \|.
    \end{equation*}
\end{lemma}

For later use, we also recall the following discrete version of Gronwall lemma, as stated in \citep[Lemma 2.18]{layton2008numerical}.
\begin{lemma}[Discrete Gronwall Lemma]
\label{lemma:discrete_gronwell}
   Let $\Delta t > 0$, $E > 0$.
   Furthermore, let $\left\{a_k\right\}_{k = 0}^K,$
   $\left\{b_k\right\}_{k = 0}^K,$ $\left\{c_k\right\}_{k = 0}^K,$ $\left\{d_k\right\}_{k = 0}^K$ be nonnegative sequences such that 
   \begin{equation}
       a_K + \Delta t \sum_{k=0}^K b_k \leq \Delta t \sum_{k=0}^K d_k a_k + \Delta t \sum_{k = 0}^K c_k + E       
   \end{equation}
   and
   \begin{equation}
   d_k < \frac{1}{\Delta t} \quad \text{ for } k = 0, \hdots, K.
   \end{equation}
   Then,
   \begin{equation}
         a_K + \Delta t \sum_{k=0}^K b_k \leq \exp  \left( \Delta t \sum_{k=0}^K \frac{d_k}{1 - \Delta t d_k} \right) \left( \Delta t \sum_{k=0}^K c_k  + E\right).
   \end{equation}
\end{lemma}

The fully discrete scheme is finally obtained by applying a futher discretization in space by means of 
the finite element method. We use the Taylor-Hood finite element spaces $\bX^h \subset \bX$ and $Q^h \subset Q$, where $\bX^h$ is the finite element space of piecewise polynomials of order $m+1$, with $m$ a positive integer, 
$Q^h$ is the finite element space of piecewise polynomials of order $m$, and $h$ denotes the mesh size.
The Taylor-Hood finite element spaces yield the following finite element discretization:
Given $\bu^h_k \in \bX^h$, find $\bu^h_{k+1} \in \bX^h, \, p^h_{k+1} \in Q^h$, such that,\footnote{The notation for the fully discrete solutions should have been $\bu^{\Delta t, h}_{k}$ and $p^{\Delta t, h}_{k}$, as a reminder that they are resulting from a time discretization with time step $\Delta t$ and space discretization with mesh size $h$. We prefer the simpler notation $\bu^{h}_{k}$ and $p^{h}_{k}$ because the semi-discrete scheme in \eqref{eq:NSE_weak_semi} is just an intermediate step towards \eqref{eq:NSE_weak_fom}, and neither the scheme nor its solutions $\bu^{\Delta t}_{k}$ and $p^{\Delta t}_{k}$ will be used 
afterwards.}
\begin{equation}
\begin{cases}
\displaystyle\int_{\Omega} \frac{3 \bu^h_{k + 1} - 4 \bu^h_{k} + \bu^h_{k-1}}{2\Delta t}\cdot \bv^h \; dx
	+ \nu (\nabla \bu^h_{k + 1},\nabla \bv^h)
	+ b^*(\bu^h_{k + 1}, \bu^h_{k + 1}, \bv^h) &
	\\
	\qquad\qquad- (\nabla \cdot \bv^h, p^h_{k + 1})
	= (\bff^h_{k + 1}, \bv^h), & \forall \bv^h \in X^h,\\
(\nabla \cdot \bu^h_{k + 1}, q^h) = 0, & \forall q^h \in Q^h.
\end{cases}
\label{eq:NSE_weak_fom}
\end{equation}

In the following, the finite element formulation in \eqref{eq:NSE_weak_fom} will be considered as the FOM scheme.
We 
also assume that the FOM scheme satisfies the following asymptotic error bound, which was also assumed in~\citep[Assumption 2.2]{xie2018numerical}:
\begin{assumption}
    We assume that the FE approximation $\bu^{h}$ satisfies the following error 
    bound: 
    \begin{eqnarray}
        \| \bu_{K} - \bu^{h}_{K} \|^{2}
        + \Delta t \; \sum_{k=1}^{K} \| \nabla (\bu_{k} - \bu^{h}_{k}) \|^{2}
        \leq C \left(
            h^{2m+2}
            + \Delta t^4
        \right).
        \label{eqn:assumption-finite-element-1}
    \end{eqnarray}
    We also assume the following standard approximation property: 
    \begin{equation}
        \inf_{q^h \in Q^h} \| p - q^h \| \leq C h^{m+1} \qquad \text{for any } p \in Q.
    \end{equation}
    \label{assumption:finite-element}
\end{assumption}
In~\eqref{eqn:assumption-finite-element-1}, and in what follows, we denote with $C$ a generic constant that can depend on the problem data, but not on the finite element parameters (e.g., $h, m, \Delta t$) or, in the context of the next section, the ROM parameters (e.g., $r, \lambda_j$). We note that it is reasonable to assume that \eqref{eqn:assumption-finite-element-1} holds since a similar error bound was indeed proved, e.g., in \citet{layton2008introduction} for the $\mathbbm{P}^{m+1} / \mathbbm{P}^{m}$ Taylor-Hood pair for the FE space discretization, and a second-order time stepping scheme; see also 
error bound (2.7) in \citet{mohebujjaman2017energy}.

\section{Reduced Order Model Preliminaries}
    \label{sec:rom-preliminaries}

\subsection{Proper Orthogonal Decomposition}

We will briefly discuss proper orthogonal decomposition, referred to as POD, which is one of the primary techniques used in Galerkin reduced order modeling to obtain a low-dimensional basis. A more thorough presentation of the topic is given in, e.g., \citet{KV01}. 
For a specific time $t_k$, we call $\bu(\bx,t_k)$ a velocity snapshot. This snapshot contains data for all $\bx \in \Omega$ at that specific time. If the exact solution $\bu$ were available, then $\bu_k(\bx)$ could be used as snapshot at time $t_k$; in practice, since the exact solution is not available, $\bu_k^h(\bx)$ is used instead.

Given an ensemble of snapshots 
$\{\bu(\bx,t_k)\}_{k=0}^K$,
the POD method obtains a low-dimensional basis, $\{\bvarphi_1, \dots \bvarphi_r\}$, that 
approximates the given data by solving the following minimization problem \citep{KV01}:
\begin{equation}
    \min_{\{\bvarphi_j\}_{j=1}^r} \frac{1}{K+1} \sum_{k=1}^K \left\| \bu(\bx,t_k) - \sum_{j=1}^r \bigg( \bu(\bx,t_k), \bvarphi_j(\bx) \bigg) \bvarphi_j(\bx) \right\|^2,
\label{eqn:POD_minimization}
\end{equation}
subject to $(\bvarphi_i,\bvarphi_j) = \delta_{ij}$, the Kronecker delta.
We introduce the correlation matrix $\mathcal{C} \in \mathbb{R}^{(K+1) \times (K+1)}$, where $\mathcal{C}_{k \ell} := \frac{1}{K+1} \left( \bu(\bx,t_k), \bu(\bx,t_\ell)\right) $ for $k, \ell = 0, \dots, K$. The solution to \eqref{eqn:POD_minimization} may be found by solving the eigenvalue problem $\mathcal{C} \bv_j = \lambda_j \bv_j$. These eigenvalues $\lambda_j$ are positive and sorted 
in descending order $\lambda_1 \geq \dots \geq \lambda_r \geq \dots \geq \lambda_R \geq 0$, where $R =$ rank($\mathcal{C}$).
Then, a set of basis functions of dimension $r \leq R$ satisfying \eqref{eqn:POD_minimization} can be shown to be given by $\bvarphi_j(\bx) = \frac{1}{\lambda_j} \sum_{k=1}^K (\bv_j)_k \bu(\bx,t_k), \; j = 1, \dots, r.$ Here, $(\bv_j)_k$ is the $k^\text{th}$ component of the $j^\text{th}$ eigenvector $\bv_j$. Importantly, choosing a basis of dimension $r < R$ leads to an approximation error equal to the sum of the discarded eigenvalues \citep{KV01}:
\begin{equation}
    \frac{1}{K+1} \sum_{k=1}^K \left\| \bu(\bx,t_k) - \sum_{j=1}^r \bigg( \bu(\bx,t_k), \bvarphi_j(\bx) \bigg) \varphi_j(\bx) \right\|^2 = \sum_{j=r+1}^R \lambda_j.
\end{equation}

The POD algorithm introduced above is now employed to define a reduced velocity space span$\{ \bvarphi_1, \dots, \bvarphi_r \} = \bX^r \subset \bX^h$. 
Finally, we introduce $S^{r}$, the ROM stiffness matrix corresponding to the diffusive term, with components $S^{r}_{ij} = \left( \nabla \bvarphi_i, \nabla \bvarphi_j \right)$. The spectral norm $\|S^{r}\|_2 = \sigma_{\max}(S^{r})$ of this matrix is important 
in the numerical analysis of the error 
introduced by POD, as illustrated next. Furthermore, it also appears in the following POD inverse inequality
\begin{equation}
\left\|\nabla \bv^r\right\| \leq \sqrt{\|S^{r} \|_2} \ \|\bv^r\| \qquad \forall \bv^r \in \bX^r,
\label{eq:POD_inv_ineq}
\end{equation}
as shown in \citep[Lemma 2 and Remark 2]{KV01}.

\subsubsection{ROM Projection}
In this section, we introduce the ROM projection and include several associated results that will be used in the ADL-ROM numerical analysis in Section~\ref{sec:adl-rom-numerical-analysis}.

\begin{definition}[ROM Projection]
    For any $\bu \in L^2(\Omega; \mathbb{R}^d)$, define $\bP^r(\bu)$ as the unique element of $\bX^r$ such that
    \begin{equation}
        \left(  \bP^r(\bu) , \bv^r \right) = (\bu, \bv^r)
        \qquad \forall \, \bv^r \in \bX^r.
    \end{equation}
    \label{definition:rom-projection}
\end{definition}

In the ADL-ROM numerical analysis in Section~\ref{sec:adl-rom-numerical-analysis}, 
for the ROM projection of the sequence $\left\{\bu_k\right\}_{k = 0}^K$ of weak
velocity solutions, we use the following lemma, which was proven in \citep[Lemma 3.3]{iliescu2014variational}:
\begin{lemma}
\label{lemma:L2_ROM_projection_error}
    The sequence of ROM projections $\left\{\bP^r(\bu_k)\right\}_{k = 0}^K$ satisfies the following inequalities:
    \begin{align}
        \frac{1}{K+1} \sum_{k=1}^K \| \bu_k - \bP^r(\bu_k) \| ^2 &\leq C \left(h^{2m+4} + \Delta t^4 + \sum_{j = r+1}^R \lambda_j\right) \label{eq:L2_ROM_projection_error_1}\\
        \frac{1}{K+1} \sum_{k=1}^K \|\nabla( \bu_k - \bP^r(\bu_k) )\| ^2 &\leq C \left(h^{2m+2} + \|S^{r} \|_2 \ h^{2m+4} + (1 + \| S^{r} \|_2) \Delta t^4 + \sum_{j = r+1}^R \| \nabla \bvarphi_j \|^2 \lambda_j \right).
        \label{eq:L2_ROM_projection_error_2}
    \end{align}
\end{lemma}

Similarly to \citep[Assumption 4.1]{xie2018numerical}, for later use in Section \ref{sec:rom_df} and in the proof of Theorem~\ref{thm:32}, we 
state 
an assumption that is stronger than the result in \eqref{eq:L2_ROM_projection_error_1}-\eqref{eq:L2_ROM_projection_error_2}.

\begin{assumption}
\label{assump:POD_approximation}
    Assume that the sequence of ROM projections $\left\{\bP^r(\bu_k)\right\}_{k = 0}^K$ satisfies
    \begin{align}
        \| \bu_k - \bP^r(\bu_k) \|^2 &\leq C \left(h^{2m+4} + \Delta t^4 + \sum_{j = r+1}^R \lambda_j\right) \label{eq:POD_approximation_1}\\
        \|\nabla( \bu_k - \bP^r(\bu_k) )\|^2 &\leq C \left(h^{2m+2} + \|S^{r} \|_2  \ h^{2m+4} + (1 + \| S^{r} \|_2) \Delta t^4 + \sum_{j = r+1}^R \| \nabla \bvarphi_j \|^2 \lambda_j \right). 
        \label{eq:POD_approximation_2}
    \end{align}

\end{assumption}

\begin{remark}
    The pointwise in time error bounds 
    in Assumption~\ref{assump:POD_approximation} are needed in the proof of Theorem~\ref{thm:32}.
    Specifically, we use \eqref{eq:POD_approximation_2} to prove 
    Lemma~\ref{lemma:assumption-u-gradu-gradeta}.
    Furthermore, as noted in the proof of 
    Lemma~\ref{lemma:assumption-u-gradu-gradeta},
    using \eqref{eq:L2_ROM_projection_error_2} instead of \eqref{eq:POD_approximation_2} to prove inequality~\eqref{eqn:32_eq20} would yield a suboptimal bound.

    The motivation for Assumption~\ref{assump:POD_approximation} was given in~\citet{koc2021optimal} (see also Remark 3.2 in~\citet{iliescu2014variational}), where we showed that using the bounds in Lemma~\ref{lemma:L2_ROM_projection_error} can yield suboptimal error bounds.
    Specifically, in~\citet{koc2021optimal} we constructed two counterexamples that showed that Assumption~\ref{assump:POD_approximation} can fail and, in that case, the error is suboptimal.

    Furthermore, in~\citet{koc2021optimal} we proved that using both the snapshots and the snapshot difference quotients to construct the ROM basis, optimal error bounds can be proved without making Assumption~\ref{assump:POD_approximation}.
    This result was further improved in~\citet{eskew2023new,garcia2023second}, where it was shown that, to prove optimal error bounds, it is sufficient to use only the snapshot difference quotients and the snapshot at the initial time or the mean value of the snapshots.
    Further improvements were recently proved in~\citet{garcia2024pointwise}.
    
    For simplicity, in this paper we do not include the snapshot difference quotients and instead assume the pointwise in time error bounds in Assumption~\ref{assump:POD_approximation}.
     
    \label{remark:assumption-dq}
\end{remark}

\subsection{Galerkin ROM}\label{sec:galerkin_rom}
The Galerkin ROM (G-ROM) employs both POD truncation and a Galerkin method to simplify complex fluid dynamics problems. POD truncation facilitates an approximation of the velocity field by expressing it as a linear combination of the truncated POD basis functions. Mathematically, this can be represented as 
\begin{equation}
\bu(\bx, t) \approx \bu^{r}(\bx, t) \equiv \sum_{j=1}^{r} a_{j}(t) \bvarphi_{j}(\mathbf{x}).
\end{equation}
Here, $\bu(\bx,t)$ is the velocity field, $\bu^r(\bx,t)$ is its ROM approximation, $\left\{ \bvarphi_j 
\right\}_{j=1}^{r}$ are the POD basis functions, and $\left\{ a_j(t)\right\}_{j=1}^{r}$ are the time-dependent coefficients that need to be determined. By substituting the ROM velocity field $\bu^r$ for the full-order velocity field $\bu$ in the 
NSE (\ref{eqn:nse-1})-(\ref{eqn:nse-2}), and applying the Galerkin method, we project the resulting equations onto the ROM space $\bX^r$. This process leads to the formulation of the G-ROM for the NSE, which can be stated as: Find $\bu^r \in \bX^r$ such that
\begin{equation}\label{eqn:grom_for}
\left(\frac{\partial \bu^{r}}{\partial t}, \bvarphi\right)+\nu\left(\nabla \bu^{r}, \nabla \bvarphi\right)+b^{*}\left(\bu^{r}, \bu^{r}, \bvarphi\right)=(\bff, \bvarphi), \quad \forall \bvarphi \in \bX^{r}
\end{equation}
and $\bu^{r}(\cdot, 0) \in \bX^{r}$. The error analysis of the spatial and temporal discretizations for the G-ROM \eqref{eqn:grom_for} has been extensively studied in the literature; see, e.g.,~\citet{chapelle2012galerkin,iliescu2014are,kostova2015error,KV01,KV02,luo2008mixed,singler2014new}. These studies address various aspects of the accuracy and stability of the ROM, providing insights into its performance under different conditions. We note that ROM velocities in $\bX^r$ are weakly-divergence free, and hence \eqref{eqn:grom_for} does not include any pressure contribution. If a ROM pressure approximation is desired, alternative ROMs are devised, e.g., in \citet{ballarin2015supremizer,decaria2020artificial,kean2020error,noack2005need,rozza2007stability,rubino2020numerical,stabile2018finite}.
In this paper, we will not include a ROM pressure approximation since the novelty of the proposed ADL-ROM lies in the different treatment of the non-linear term of the NSE.

Employing the same time stepping scheme as in the FOM, the following 
algorithm summarizes the BDF2 G-ROM scheme:

\begin{algorithmplain}[BDF2 G-ROM Scheme]\label{alg:BDF2_scheme_G}
    For $\Delta t > 0$ and for $k = 1, \dots, K-1$, given $\bu_{k-1}^r,\bu_{k}^r \in \bX^r$, find $\bu^r_{k+1} \in \bX^r$ such that
    \begin{align}\label{eqn:BDF2_eq_G}
        \frac{\left(3\bu_{k+1}^r - 4 \bu_{k}^r + \bu_{k-1}^r, \bv^r\right)}{2\Delta t}
        + \nu \left(\nabla \bu_{k+1}^r,\nabla \bv^r\right)
        + b^*\left(\bu_{k+1}^r, \bu_{k+1}^r, \bv^r\right)  \nonumber  \\
        = \left(\bff_{k+1}, \bv^r\right) \;\; \forall \bv^r \in \bX^r.
    \end{align}
\end{algorithmplain}

Despite its significant computational efficiency due to POD truncation, the G-ROM \eqref{eqn:BDF2_eq_G} has generally been limited to modeling laminar flows due to its inherent assumptions and simplifications. To extend the applicability of the ROM to more complex and turbulent flow scenarios, we consider the development of the approximate deconvolution Leray ROM, which we will build in the remaining part of this section.

\subsection{ROM Differential Filter}\label{sec:rom_df}

The regularization scheme we employ is based on differential filtering. In this work, we use the following strong form of the differential filter. 
\begin{definition}[Continuous Differential Filter]
    Let $\delta \in \mathbb{R}^+$ be a constant filter radius. Given $\bphi \in C^0(\Omega; \ \mathbb{R}^d)$, the continuous differential filter is a map 
    $$F: C^0(\Omega; \ \mathbb{R}^d) \to C^2(\Omega; \ \mathbb{R}^d), \quad \bphi \mapsto \obphi$$ such that $\obphi \in C^2(\Omega; \ \mathbb{R}^d)$ is the unique solution to
    \begin{equation}
    \label{EQ:continuous_filter}
    \begin{cases}
    (I - \delta^2 \Delta) \obphi = \bphi, & \text{in } \Omega,\\
    \obphi = \bzero, & \text{on }\partial\Omega.
    \end{cases}
    \end{equation}
    \label{def:filter_definition}
\end{definition}

The continuous differential filter as described above is impractical in combination with the FOM or the ROM, because it is based on a strong formulation of the operator $(I - \delta^2 \Delta)$. Instead, a weak formulation of $(I - \delta^2 \Delta)$ will be used in the definition of the following ROM differential filter.

\begin{definition}[ROM Differential Filter]
    Let $\delta \in \mathbb{R}^+$ be a constant filter radius. Given $\bphi \in L^2(\Omega; \ \mathbb{R}^d)$, the ROM differential filter is a map 
    $$F^r: L^2(\Omega; \ \mathbb{R}^d) \to \bX^r, \quad \bphi \mapsto \obphi$$ such that $\obphi \in \bX^r$ is the unique solution to
    \begin{equation}
    \label{EQ:discrete_filter}
    \delta^2(\nabla \obphi, \nabla \bv^r) + (\obphi, \bv^r) = (\bphi,\bv^r) \quad \forall \bv^r \in \bX^r.
    \end{equation}
\end{definition}
The symbol $\obullet$ will be reserved in the following as 
shorthand for the application of $F^r 
$ to the input $\bullet$. The typical input will be a velocity field $\bu \in \bX \subset L^2(\Omega; \ \mathbb{R}^d)$. We further highlight that $\obullet$ does depend on $r$ through the reduced velocity space $\bX^r$. 
To simplify the notation, we will not add a further symbol $r$ to the overline bar. 

Lemma \ref{thm:Xie_Lemma212}, which we present below, was proven in \citep[Lemma 4.3]{xie2018numerical} based on Assumption \ref{assump:POD_approximation}. We note that, in the original lemma in \citet{xie2018numerical}, 
the LHS included an additional addend, 
which we drop 
since it is not used in our analysis.
The bound in Lemma \ref{thm:Xie_Lemma212} categorizes the sources of error in the filtering process: those which belong to the finite element data, those which belong to the POD process, those which depend on the filtering process, and mixed terms. 

\begin{lemma}
\label{thm:Xie_Lemma212}
    Let $\bu_k = \bu(\bx,t_k) \in \bX$ be a solution to~\eqref{eq:NSE_weak} with $\Delta \bu_k\in L^2(\Omega).$ 
    If Assumption \ref{assump:POD_approximation} holds, then, for all $k$,  

\begin{align*}
    \| \bu_k - \obu_k \|^2 &\leq  C\left(h^{2m+4} + \Delta t^4 + \sum_{j=r+1}^R \lambda_j\right) \\
    &+ C \delta^2 \left( h^{2m+2} + \| S^{r} \|_2 \ h^{2m+4} +  (1 + \|S^{r} \|_2) \ \Delta t^4 + \sum_{j=r+1}^R \| \nabla \bvarphi_j \|^2 \lambda_j \right) \\
    &+ C \delta^4 \| \Delta \bu_k \|^2.
\end{align*}
\end{lemma}

We further recall the following estimates from \citep[Lemma 4.4]{xie2018numerical}:
\begin{lemma}\label{thm:ROM211_split1}
Let $\bu \in \bX$. Then

\begin{align}
    \left\|\obu\right\| & \leq\|\bu\|, \label{eqn:Lem211_f1}\\
    \left\|\nabla \obu\right\| & \leq \sqrt{\|S^{r} \|_2} \ \|\bu\|.\label{eqn:Lem211_f2}
\end{align}
Furthermore, if $\bu \in \bX^r$, then
\begin{align}
\left\|\nabla \obu\right\| & \leq\|\nabla \bu\|. \label{eqn:Lem211_f3}
\end{align}

\end{lemma}

\begin{remark}
    We note that the a priori stability bounds in Lemma~\ref{thm:ROM211_split1}, which were proven in \citep[Lemma 4.4]{xie2018numerical}, are different from the FE stability bounds in \citep[Lemma 2.11]{layton2008numerical}.

    Specifically, in the ROM setting, we prove two different bounds in Lemma~\ref{eqn:Lem211_f1} (see Lemma 4.4 in~\citet{xie2018numerical}): 
    (i) for $\bu \in \bX$, we prove the bound in~\eqref{eqn:Lem211_f2}, and 
    (ii) for $\bu \in \bX^r$, we prove the stronger bound in~\eqref{eqn:Lem211_f3}.

    In contrast, in the FEM setting, the stronger bound (corresponding to \eqref{eqn:Lem211_f3}) is proven in inequality (2.22) in~\citet{layton2008numerical} for $\bu \in \bX$.
    We note that, to prove inequality (2.22) in~\citet{layton2008numerical}, the authors use the Ritz projection (William Layton, personal communication).
    Since the approximation properties of the Ritz projection in a POD setting are not clear~\citep{iliescu2014are,koc2021optimal}, extending the strategy in~\citet{layton2008numerical} to the ROM setting and proving the bound in~\eqref{eqn:Lem211_f3} for $\bu \in \bX$ is an open problem.
    \label{remark:filter-stability-bounds}
\end{remark}

We note that \eqref{eqn:Lem211_f1} suggests to define an auxiliary operator $\widetilde{F}^r: \bX \to \bX^r$ such that $\widetilde{F}^r := F^r|_{\bX}$. With this notation, \eqref{eqn:Lem211_f1} implies that $\widetilde{F}^r$ is a bounded operator on $\bX$, while no similar conclusion can be stated in general for $F^r$ on $L^2(\Omega; \mathbb{R}^d) \supset \bX$.

\subsection{ROM Approximate Deconvolution}\label{sec:rom_ad}

To construct the approximate deconvolution Leray ROM, we need to define a ROM approximate deconvolution (AD) operator.
The main purpose of the AD operator is to take an overly smoothed solution component and produce a more accurate approximation of the solution.
The AD idea, which has been extensively used in image processing and inverse problems, has been extended to ROM closures in~\citet{xie2017approximate} and to ROM stabilizations in~\citet{sanfilippo2023approximate}.
The AD operator that we use in this paper is 
the van Cittert operator, 
which was also utilized in \citet{layton2008numerical} at a finite element level. Alternatives could be classical methods, such as Tikhonov or Lavrentiev regularization, which have been successfully employed in image processing and large eddy simulation 
\citep{bertero1998introduction,layton2012approximate}, 
or more recent regularization approaches 
\citep{Benning2018inverse}.

\begin{definition}[Continuous van Cittert operator]
The continuous van Cittert operator $D_N: C^0(\Omega; \ \mathbb{R}^d) \to C^0(\Omega; \ \mathbb{R}^d)$ 
is defined as the map
\begin{equation}
\label{EQ:continuous_Cittert}
    \bphi \mapsto D_N \bphi := \sum_{n=0}^N (I - F)^n \bphi.
\end{equation}
\end{definition}
\begin{definition}[ROM van Cittert operator]
The ROM van Cittert operator $D_N^r: L^2(\Omega; \ \mathbb{R}^d) \to \bX^r$ is defined as the map
\begin{equation}
\label{EQ:Discrete_Cittert}
    \bphi \mapsto D_N^r \bphi := \sum_{n=0}^N (I - F^r)^n \bphi.
\end{equation}
\end{definition}

We now extend several lemmas from the FOM finite element setting of \citet{layton2008numerical} to a ROM setting,  and discuss other relevant ROM lemmas. The first lemma we extend is \citep[Lemma 2.11]{layton2008numerical}: 

\begin{lemma}\label{thm:ROM211_split2}
Let $\bu \in \bX$. Then
\begin{align}
    \left\|D_N^r \obu\right\| & \leq C(N) \|\bu\| \label{eqn:Lem211_s1}\\
    \left\|\nabla D_N^r \obu\right\| & \leq C(N) \sqrt{\|S^{r} \|_2} \ \|\bu\|. \label{eqn:Lem211_s2}
\end{align}
Furthermore, if $\bu \in \bX^r$, then
\begin{align}
    \left\|\nabla D_N^r \obu\right\| & \leq C(N) \|\nabla \bu\|. \label{eqn:Lem211_s3}
\end{align}
\end{lemma}

Similarly to Lemma \ref{thm:ROM211_split1}, \eqref{eqn:Lem211_s1} implies that the composition $\widetilde{D}_N^r \circ \widetilde{F}^r$ is a bounded operator on $\bX$, where $\widetilde{D}_N^r := D_N^r|_{\bX^r}$ is the restriction of $D_N^r$ to $\bX^r$.

\begin{proof}
Since the uniform in $N$ boundedness of $\widetilde{F}^r$ is guaranteed by \eqref{eqn:Lem211_f1}, it is clear that to prove \eqref{eqn:Lem211_s1} we only need to prove that $\widetilde{D}_N^r$ is a bounded operator from $\bX^r$ to itself. 
When $\bphi \in \bX^r$, $F^r$ in \eqref{EQ:Discrete_Cittert} can be replaced by $\widetilde{F}^r$. Upon such replacement, the right-hand side of \eqref{EQ:Discrete_Cittert} contains linear combinations and compositions of the continuous operator $I - \widetilde{F}^r$, 
implying that $\widetilde{D}_N^r$ 
is continuous as well, with a continuity constant that depends on $N$.

To prove \eqref{eqn:Lem211_s2}, we use the POD inverse inequality \eqref{eq:POD_inv_ineq} and the bound in~\eqref{eqn:Lem211_s1}:
\begin{equation}
    \|\nabla D_N^r \obu\| \stackrel{\eqref{eq:POD_inv_ineq}}{\leq} \sqrt{\|S^{r} \|_2} \ \| D_N^r \obu\| \stackrel{\eqref{eqn:Lem211_s1}}{\leq} C(N) \sqrt{\|S^{r} \|_2} \ \|\bu\|.
\end{equation}

Upon recognizing that $D_N^r$ can be replaced by $\widetilde{D}_N^r$ in \eqref{eqn:Lem211_s3}, we finally prove \eqref{eqn:Lem211_s3} by induction. 
The base case $N = 0$ follows from $\eqref{eqn:Lem211_f3}$, and results in $C(0) = 1$. 
For the induction step, assume that \eqref{eqn:Lem211_s3} holds for $N - 1$. 
We first rewrite $\widetilde{D}_N^r$ as
\begin{align}
    \widetilde{D}_N^r \bphi &= \sum_{n=0}^N (I - \widetilde{F}^r)^n \bphi
    = \bphi + \sum_{n=1}^N (I - \widetilde{F}^r)^n \bphi\notag\\
    &= \bphi + (I - \widetilde{F}^r) \sum_{n=0}^{N - 1} (I - \widetilde{F}^r)^n \bphi
    = (I + (I - \widetilde{F}^r) \widetilde{D}_{N-1}^r) \bphi\label{eq:VC_induc}
\end{align}
to obtain, in particular, 
\begin{equation*}
    \widetilde{D}_N^r \obu = \obu + \widetilde{D}_{N-1}^r \obu - \overlineIMPLEMENTATION{\widetilde{D}_{N-1}^r \obu}, 
\end{equation*}
and thus
\begin{align*}
    \|\nabla \widetilde{D}_N^r \obu\| 
    &\stackrel{\substack{\text{triangle}\\\text{inequality}}}{\leq}
    \|\nabla \obu\|
    + \|\nabla \widetilde{D}_{N-1}^r \obu \|
    + \|\nabla \overlineIMPLEMENTATION{\widetilde{D}_{N-1}^r \obu} \|\\
    &\stackrel{\eqref{eqn:Lem211_f3}}{\leq}
    \|\nabla \bu\|
    + 2 \|\nabla \widetilde{D}_{N-1}^r \obu \|
    \stackrel{\substack{\eqref{eqn:Lem211_s3} \text{ for } N - 1}}{\leq}
    (1 + 2 C(N - 1))\ \|\nabla \bu \|,
\end{align*}
which concludes the proof.
\end{proof}

Using Lemma \ref{thm:Xie_Lemma212}, we can also extend \citep[Lemma 2.13]{layton2008numerical} as follows:
\begin{lemma}
    \label{lemma:Layton_213_Rom}
Let $\bu_k = \bu(\bx,t_k) \in \bX$ be a solution to~\eqref{eq:NSE_weak} with $\Delta \bu_k\in L^2(\Omega).$ 
    If Assumption \ref{assump:POD_approximation} holds, then, for all $k$

\begin{align*}
     \| \bu_k - D_N^r \obuk{k} \|^2 &\leq C(N) \left(h^{2m+4} + \Delta t^4 + \sum_{j=r+1}^R \lambda_j\right) \\
    &+ C(N) \delta^2 \left( h^{2m+2} + \| S^{r} \|_2 \ h^{2m+4} +  (1 + \|S^{r} \|_2) \Delta t^4 + \sum_{j=r+1}^R \| \nabla\bvarphi_j \|^2 \lambda_j \right)\\
  &+ C(N) \delta^4 \| \Delta \bu_k \|^2.
\end{align*}
\end{lemma}
\begin{proof}
First, we recall that from \eqref{eq:VC_induc} the computation of $D_N^r \obu$ can be equivalently carried out by means of the following iterative process
\begin{equation}
\begin{cases}
\bu^{(0)} = \obu := \widetilde{F}^r \bu,&\\
\bu^{(n + 1)} = \widetilde{F}^r \bu + (I - \widetilde{F}^r) \bu^{(n)}, & n = 0, \hdots, N - 1.
\end{cases}
\label{eq:recurrence_Layton_213_Rom}
\end{equation}
We wish now to obtain a recurrence formula for the AD error
\begin{equation*}
\be^{(N)} := D_N^r \obu - \bu = \bu^{(N)} - \bu.
\end{equation*}
To do so, replace $\widetilde{F}^r \bu$ in \eqref{eq:recurrence_Layton_213_Rom} by exploiting the identity
\begin{equation*}
 \widetilde{F}^r \bu = \bu - (I - \widetilde{F}^r) \bu
 \end{equation*}
to 
obtain
\begin{equation}
\begin{cases}
\be^{(0)} = \obu - \bu = - (I - \widetilde{F}^r) \bu,&\\
\be^{(n + 1)} = (I - \widetilde{F}^r) \be^{(n)} 
, & n = 0, \hdots, N - 1,
\end{cases}
\label{eq:recurrence_error_Layton_213_Rom}
\end{equation}
and hence, in particular,
\begin{equation}
\be^{(N)} = \left(I - \widetilde{F}^r\right)^{N} (\obu - \bu).
\end{equation}

We now consider $\left(I - \widetilde{F}^r\right)^N$ as an operator on $\bX$, and show that
\[ \left\| \left(I - \widetilde{F}^r\right)^N \right\| \leq \left\|I - F^r \right\|^N \leq \left(\|I \| + \| F^r \| \right)^N \leq C(N),
\]
by standard properties of operator norms and the triangle inequality, 
where the final inequality is due to the fact that $\widetilde{A}_r$ is a bounded operator on $\bX$, as shown by \eqref{eqn:Lem211_f1}.
Using this, we obtain
\[
\| D_N^r \obu - \bu \| \leq C(N) \|\obu - \bu \|,
\]
and conclude by bounding the right-hand side with Lemma \ref{thm:Xie_Lemma212}.

\end{proof}

\subsection{Approximate Deconvolution Leray ROM (ADL-ROM)}\label{sec:rom_adl}

One of the most popular regularized ROMs is the {\emph Leray ROM (L-ROM)}~\citep{Girfoglio2019CF,girfoglio2021pod,kaneko2020towards,sabetghadam2012alpha,wells2017evolve}, which 
is based on Jean Leray's 
simple but powerful idea 
to replace one part of the non-linear term in the Navier-Stokes equations, $\bu \cdot \nabla \bu$, with a filtered term with higher regularity, 
i.e., $\obu \cdot \nabla \bu$ \citep{leray1934sur}. 
Although the Leray ROM has been successfully used in the numerical simulation of convection-dominated flows, one of its drawbacks is that it can be overdiffusive (e.g., when the fiter radius is too large; see, e.g.,~\citet{sanfilippo2023approximate}).
To address this drawback of the L-ROM, in~\citet{sanfilippo2023approximate} we  proposed AD as a means to increase the L-ROM accuracy without compromising its numerical stability.
Using the same time stepping scheme as that used in the FOM (i.e., BDF2), the new {\emph approximate deconvolution Leray ROM (ADL-ROM)}~\citep{sanfilippo2023approximate} can be summarized in the following algorithm:

\begin{algorithmplain}[BDF2 ADL-ROM Scheme]\label{alg:BDF2_scheme}
    For $\Delta t > 0$ and for $k = 1, \dots, K-1$, given $\bu_{k-1}^r,\bu_{k}^r \in \bX^r$, find $\bu^r_{k+1} \in \bX^r$ such that
    \begin{align}\label{eqn:BDF2_eq}
        \frac{\left(3\bu_{k+1}^r - 4 \bu_{k}^r + \bu_{k-1}^r, \bv^r\right)}{2\Delta t}
        + \nu \left(\nabla \bu_{k+1}^r,\nabla \bv^r\right)
        + b^*\left(D_N^r\oburk{k+1}, \bu_{k+1}^r, \bv^r\right) \nonumber  \\
        = \left(\bff_{k+1}, \bv^r\right) \;\; \forall \bv^r \in \bX^r.
    \end{align}
\end{algorithmplain}

\section{ADL-ROM Numerical Analysis}
    \label{sec:adl-rom-numerical-analysis}

The following lemma on the {\emph{a priori}} bound, or stability estimate, for the approximation scheme in (\ref{eqn:BDF2_eq}) extends \citep[Lemma 4.1]{mohebujjaman2017energy} to the BDF2 ADL-ROM scheme in Algorithm \ref{alg:BDF2_scheme}.

\begin{lemma}\label{lemma:a_priori_bound_BDF2}
    The approximation scheme in Algorithm \ref{alg:BDF2_scheme} has a solution $\bu_k^r$ at each timestep $k = 2, \dots, K$ which satisfies the following \emph{a priori} bound:
    \begin{align}
        \label{eqn:a_priori_BDF2}
        \left\|\bu^{r}_K\right\|^{2}+2 \nu \Delta t \sum_{\kappa=1}^{K-1}\left\|\nabla \bu^{r}_{\kappa+1}\right\|^{2} \leq\left\|\bu^{r}_{1}\right\|^{2}+\left\|2 \bu^{r}_{1}-\bu^{r}_{0}\right\|^{2}+2 \nu^{-1} \Delta t \sum_{\kappa=1}^{K-1}\left\|\bff_{\kappa+1}\right\|^{2} \nonumber \\
        =C\left(\bu^{r}_{0}, \bu^{r}_{1}, \bff, \nu^{-1}\right).
    \end{align}
\end{lemma}

\begin{proof}
    The proof is 
    similar to the one of \citep[Lemma 4.1]{mohebujjaman2017energy}. It entails choosing $\bv^r=\bu^r_{k+1}$ in \eqref{eqn:BDF2_eq} 
    and noting that $b^*(D_N^r \oburk{k+1}, \bu^r_{k+1}, \bu^r_{k+1}) = 0$ by Lemma \ref{Lemma:skew_Symmetric}. Thus, we obtain
    \begin{equation}
        \frac{1}{2}\left( 3\bu_{k+1}^{r} - 4 \bu_{k}^{r} + \bu_{k-1}^{r}, \bu_{k+1}^{r} \right) + \nu \Delta t \left( \nabla \bu_{k+1}^r, \nabla \bu_{k+1}^r \right) = \Delta t \left( \bff_{k+1}, \bu_{k+1}^r \right).
    \end{equation}
    The proof follows by using the Cauchy-Schwarz and Young inequalities, as well as the algebraic identity used in the proof of \citep[Lemma 4.1]{mohebujjaman2017energy}, i.e.,
    \begin{equation}\label{eqn:algid}
        \frac{1}{2}(3 a-4 b+c) a=\frac{1}{4}\left[a^{2}+(2 a-b)^{2}\right]-\frac{1}{4}\left[b^{2}+(2 b-c)^{2}\right]+\frac{1}{4}(a-2 b+c)^{2}.
    \end{equation}
\end{proof}

The following lemma is critical in ensuring the optimality of the error bound in Theorem~\ref{thm:32}.
\begin{lemma}
    For any $k = 1, \ldots, K-1$, the following inequality holds: 
    \begin{eqnarray}
        && \Delta t \sum_{k=1}^{K-1}\left\|\bu^{r}_{k+1}\right\|\left\|\nabla \bu^{r}_{k+1} \right\|\left\|\nabla \left( \bu_{k+1} - \bP^r(\bu_{k+1}) \right) \right\|^{2}
        \nonumber \\
        && \hspace*{1.0cm} 
        \leq 
        C \biggl( h^{2m+2} + h^{2m+4} \left\| S^{r} \right\|_{2} 
        + \Delta t^4 \left( 1+\left\|S^{r} \right\|_{2} \right) 
        + \sum_{j=r+1}^{R}\left\|\nabla\bvarphi_{j}\right\|^{2} \lambda_{j} \biggr),
        \label{eqn:lemma-assumption-u-gradu-gradeta-1}
    \end{eqnarray}
    where $\bP^r$ is the ROM projection in Definition~\ref{definition:rom-projection}.
    \label{lemma:assumption-u-gradu-gradeta}
\end{lemma}

\begin{proof}
To prove~\eqref{eqn:lemma-assumption-u-gradu-gradeta-1}, we use Lemma \ref{lemma:a_priori_bound_BDF2} and Assumption \ref{assump:POD_approximation}: 
    \begin{align}\label{eqn:lemma-assumption-u-gradu-gradeta-2}
        & \Delta t \sum_{k=1}^{K-1}\left\|\bu^{r}_{k+1}\right\|\left\|\nabla \bu^{r}_{k+1}\right\|\left\|\nabla \left( \bu_{k+1} - \bP^r(\bu_{k+1}) \right) \right\|^{2}  \nonumber\\
        & \quad \stackrel{\eqref{eqn:a_priori_BDF2}}{\leq} \widetilde{C} \Delta t\sum_{k=1}^{K-1} \left\|\nabla \bu^{r}_{k+1}\right\| \left\|\nabla \left( \bu_{k+1} - \bP^r(\bu_{k+1}) \right) \right\|^{2} \nonumber\\
        & \quad \leq \widetilde{C} \max_{k=0, \hdots, K-1} \left\|\nabla \left( \bu_{k+1} - \bP^r(\bu_{k+1}) \right) \right\|^{2} \Delta t\sum_{k=1}^{K-1} \left\|\nabla \bu^{r}_{k+1}\right\| \nonumber\\
        & \quad \stackrel{(\ref{eqn:a_priori_BDF2})}{\leq} C \max_{k=0, \hdots, K-1}\left\|\nabla \left( \bu_{k+1} - \bP^r(\bu_{k+1}) \right) \right\|^{2} \nonumber \\
        & \quad \stackrel{(\ref{eq:POD_approximation_2})}{\leq} C\left(h^{2m+2} + h^{2m+4}\left\|S^{r}\right\|_{2} + \Delta t^4\left(1+\left\|S^{r}\right\|_{2}\right) + \sum_{j=r+1}^{R}\left\|\nabla\bvarphi_{j}\right\|^{2} \lambda_{j}\right).
    \end{align}
    In the first inequality in~\eqref{eqn:lemma-assumption-u-gradu-gradeta-2}, we employ the first term on the left-hand side of the stability estimate \eqref{eqn:a_priori_BDF2} in Lemma \ref{lemma:a_priori_bound_BDF2} to obtain that $\left\|\bu^{r}_{k+1}\right\|$ is bounded by a constant $\widetilde{C}$ for every $k = 
    1, \hdots, K - 1$. In the second inequality, we pull the maximum of $\left\|\nabla  \left( \bu_{k+1} - \bP^r(\bu_{k+1}) \right) \right\|^{2}$ out of the sum, in order to be able to use again \eqref{eqn:a_priori_BDF2} on $\Delta t\sum_{k=1}^{K-1} \left\|\nabla \bu^{r}_{k+1}\right\|$ next. In the third inequality, we estimate $\Delta t \sum_{k=1}^{K-1} \left\|\nabla \bu^{r}_{k+1}\right\|$ using the Cauchy-Schwarz inequality first, and then the second addend on the left-hand side of \eqref{eqn:a_priori_BDF2}, namely: 
    \begin{align*}
    \Delta t \sum_{k=1}^{K-1} & \left\|\nabla \bu^{r}_{k+1}\right\| \leq \Delta t \ \widehat{C} \left(\sum_{k=1}^{K-1} \left\|\nabla \bu^{r}_{k+1}\right\|^2\right)^{\frac{1}{2}} \ K^{\frac{1}{2}}\\
    &\leq \Delta t^{\frac{1}{2}} \ \widehat{\widehat{C}} \left(\Delta t \sum_{k=1}^{K-1} \left\|\nabla \bu^{r}_{k+1}\right\|^2\right)^{\frac{1}{2}} \ \frac{1}{\Delta t^{\frac{1}{2}}} \leq \widehat{\widehat{C}} \widetilde{C} = C.
    \end{align*}

    In the fourth and final inequality in \eqref{eqn:lemma-assumption-u-gradu-gradeta-2}, we use \eqref{eq:POD_approximation_2} in Assumption \ref{assump:POD_approximation} to conclude. 
    We 
    emphasize that we would have obtained a suboptimal estimate if we had used \eqref{eq:L2_ROM_projection_error_2} instead of \eqref{eq:POD_approximation_2} in Assumption \ref{assump:POD_approximation}. 
    Indeed, multiplying both sides of \eqref{eq:L2_ROM_projection_error_2} by $K + 1$
    \begin{align*}
        & \|\nabla( \bu_k - \bP^r(\bu_k) )\| ^2 \\
        &\leq \sum_{k=1}^K \|\nabla( \bu_k - \bP^r(\bu_k) )\| ^2\\
        &\leq C \left(h^{2m+2} + \|S^{r}\|_2 \ h^{2m+4} + (1 + \| S^{r}\|_2) \Delta t^4 + \sum_{j = r+1}^R \| \nabla \bvarphi_j \|^2 \lambda_j \right) \ (K + 1), 
    \end{align*}
    which would have multiplied \eqref{eqn:32_eq20} by $K + 1$, which scales as $(\Delta t)^{-1}$ and would have caused the suboptimality.
    \end{proof}

We are now in position to state the main result of this work, which is an extension to ADL-ROM of \citep[Theorem 3.2]{layton2008numerical} and \citep[Theorem 4.1]{mohebujjaman2017energy}.

\begin{theorem}\label{thm:32}
    Let $(\bu, p)$ be a strong solution of the NSE \eqref{eqn:nse-1}-\eqref{eqn:nse-2} with homogeneous Dirichlet boundary conditions\footnote{The results can be extended to the inhomogeneous case with some additional work, but without major difficulties.} for $\bu$.
    Assume the initial conditions are $\bu_0^r=\bP^{r}\left(\bu_{0}\right)$ and $\bu_1^r=\bP^{r}\left(\bu_{1}\right)$. Then, under the regularity assumption on the exact solution (Assumption \ref{ass:reg_ex_sol}), the assumption on the FE approximation (Assumption \ref{assumption:finite-element}), and the assumption on the ROM projection error (Assumption \ref{assump:POD_approximation}), there exists $\Delta t^{*}>0$ such that, for all $\Delta t<\Delta t^{*}$, the following bound holds: 
    \begin{align}
    & \left\|\bu_{K}-\bu^{r}_{K}\right\|^{2}+ \Delta t \sum_{k=1}^{K-1}\left\|\nabla\left(\bu_{k+1}-\bu^{r}_{k+1}\right)\right\|^{2} \nonumber\\
    \leq & C \mathcal{G}(\delta, \Delta t, h, \|S^{r}\|_2,\{\lambda_j\}_{j=r+1}^{R},\{\|\nabla \bvarphi_j\|\}_{j=r+1}^{R}),
    \label{eqn:theorem-bound}
    \end{align}
    where $C=C(N,\bu_0^r,\bu_1^r,\bff,\nu)$ and $\mathcal{G}$ is defined as
    \begin{align}\label{eqn:f_32}
        &\mathcal{G}(\delta, \Delta t, h, \|S^{r}\|_2,\{\lambda_j\}_{j=r+1}^{R},\{\|\nabla\bvarphi_j\|\}_{j=r+1}^{R}) \nonumber \\
        & = 
        \left( 1 + \left\|S^{r}\right\|_2^{\frac{1}{2}} \right) \left(h^{2m+4}+\Delta t^4+{\sum_{j=r+1}^{R} \lambda_{j}}\right)  \nonumber \\
        & + \left( 1 + \delta^2 \left\|S^{r}\right\|_2^{\frac{1}{2}} \right) \left(h^{2m+2} + \|S^{r}\|_{2} \ h^{2m+4} + (1+\|S^{r}\|_{2})\Delta t^4 + \sum_{j=r+1}^{R}\|\nabla \bvarphi_j\|^{2}\lambda_j\right)\nonumber\\
        &+ \delta^4 \left\|S^{r}\right\|_2^{\frac{1}{2}}.
    \end{align}
\end{theorem}

\begin{proof}
    We begin with the weak formulation of the NSE given in \eqref{eq:NSE_weak}. At a continuous level, recall that for $\bv, \bw, \bu \in \bX$,
    \begin{equation*}
        b^{*}(\bu, \bv, \bw)=b(\bu, \bv, \bw):=(\bu \cdot \nabla \bv, \bw).
    \end{equation*}

    Then, at time $t_{k+1}$, 
    the NSE in \eqref{eq:NSE_weak} with $\bv= 
    \bv^r \in \bX^r$ 
    yield 
    \begin{equation}\label{eqn:32_eq1}
    \left(\bu_{t}^{k+1}, 
    \bv^r\right)+b^{*}\left(\bu_{k+1}, \bu_{k+1}, 
    \bv^r\right)+\nu\left(\nabla \bu_{k+1}, \nabla 
    \bv^r\right)-\left(p_{k+1}, \nabla \cdot 
    \bv^r\right)=\left(\bff_{k+1}, 
    \bv^r\right).
    \end{equation}

    Subtracting (\ref{eqn:BDF2_eq}) 
    from (\ref{eqn:32_eq1}), we obtain
    \begin{gather}\label{eqn:32_eq2}
        \left(\bu_{t}^{k+1}-\frac{3 \bu^{r}_{k+1}-4 \bu^{r}_{k}+\bu^{r}_{k-1}}{2 \Delta t}, 
        \bv^r\right) + b^{*}\left(\bu_{k+1}, \bu_{k+1}, 
        \bv^r\right) - b^{*}\left(D_{N}^{r}\oburk{k+1}, \bu^{r}_{k+1}, 
        \bv^r\right) \nonumber \\
        + \nu\left(\nabla \bu_{k+1} - \nabla \bu^{r}_{k+1}, \nabla 
        \bv^r \right) - \left(p_{k+1}, \nabla \cdot 
        \bv^r\right)=0.
    \end{gather}

    Adding and subtracting $\left(\dfrac{3 \bu_{k+1}-4 \bu_{k}+\bu_{k-1}}{2 \Delta t}, 
    \bv^r\right)$, we have
    {
    \begin{align}\label{eqn:32_eq3}
        &\left(\bu_{t}^{k+1}-\frac{3 \bu_{k+1}-4 \bu_{k}+\bu_{k-1}}{2 \Delta t}, 
        \bv^r\right) + \left(\frac{3 \bu_{k+1}-4 \bu_{k}+\bu_{k-1}}{2 \Delta t}-\frac{3 \bu^{r}_{k+1}-4 \bu^{r}_{k}+\bu^{r}_{k-1}}{2 \Delta t}, 
        \bv^r\right) \nonumber \\
        &\quad +b^{*}\left(\bu_{k+1}, \bu_{k+1}, 
        \bv^r\right) - b^{*}\left(D_{N}^{r}\oburk{k+1}, \bu^{r}_{k+1}, 
        \bv^r\right) + \nu\left(\nabla \bu_{k+1}-\nabla \bu^{r}_{k+1}, \nabla 
        \bv^r\right) = \left(p_{k+1}, \nabla \cdot 
        \bv^r\right).
    \end{align}}

    We decompose the error as follows: 
    \begin{equation*}
        \bu_{k+1}-\bu^{r}_{k+1}=\left(\bu_{k+1}-\bw^{r}_{k+1}\right) - \left(\bu^{r}_{k+1}-\bw^{r}_{k+1}\right) = \bfeta_{k+1}-\bPhi^{r}_{k+1},
    \end{equation*}
    where $\bw^{r}_{k+1}=\bP^{r}\left(\bu_{k+1}\right)$ is the $L^{2}$ projection of $\bu_{k+1}$ on $\bX^r$, i.e.,
    \begin{equation}\label{eqn:32_eq4}
        \left(\bfeta_{k+1}, \bvarphi^{r}\right) = \left(\bu_{k+1}-\bw^{r}_{k+1}, \bvarphi^{r}\right)=0, \quad\forall \bvarphi^{r} \in \bX^{r}.
    \end{equation}

    Choosing $\bv^r 
    =\bPhi^{r}_{k+1}$ and letting $\br_{k+1}=\bu_{t}^{k+1}-\dfrac{3 \bu_{k+1}-4 \bu_{k}+\bu_{k-1}}{2 \Delta t}$, (\ref{eqn:32_eq3}) becomes
    \begin{gather}
        \frac{1}{2 \Delta t}\left(3 (\bfeta_{k+1}- \bPhi^{r}_{k+1})-4\left(\bfeta_{k}-\bPhi^{r}_{k}\right)+(\bfeta_{k-1}-\bPhi^{r}_{k-1}), \bPhi^{r}_{k+1}\right) - \nu\left(\nabla \bPhi^{r}_{k+1}, \nabla \bPhi^{r}_{k+1}\right) \nonumber \\
        + \nu\left(\nabla \bfeta_{k+1}, \nabla \bPhi^{r}_{k+1}\right) + b^{*}\left(\bu_{k+1}, \bu_{k+1}, \bPhi^{r}_{k+1}\right) - b^{*}\left(D_{N}^{r}\oburk{k+1}, \bu^{r}_{k+1}, \bPhi^{r}_{k+1}\right) \nonumber \\
        + \left(\br_{k+1}, \bPhi^{r}_{k+1}\right)=\left(p_{k+1}, \nabla \cdot \bPhi^{r}_{k+1}\right). \label{eqn:32_eq5}
    \end{gather}

    Since $\bPhi^{r}_{k+1} \in \bX^{r}$, \eqref{eqn:32_eq4} implies that $\left(\bfeta_\kappa, \bPhi^{r}_{k+1}\right)=0$ for $\kappa=k-1, k, k+1$. Thus, \eqref{eqn:32_eq5} becomes 
    \begin{gather}
        \frac{1}{2 \Delta t}\left(3 \bPhi^{r}_{k+1}-4\bPhi^{r}_{k}+\bPhi^{r}_{k-1}, \bPhi^{r}_{k+1}\right) + \nu\left(\nabla \bPhi^{r}_{k+1}, \nabla \bPhi^{r}_{k+1}\right) \nonumber \\
        = \left(\br_{k+1}, \bPhi^{r}_{k+1}\right) + \nu\left(\nabla \bfeta_{k+1}, \nabla \bPhi^{r}_{k+1}\right) + b^{*}\left(\bu_{k+1}, \bu_{k+1}, \bPhi^{r}_{k+1}\right) \nonumber \\
        - b^{*}\left(D_{N}^{r}\oburk{k+1}, \bu^{r}_{k+1}, \bPhi^{r}_{k+1}\right) - \left(p_{k+1}, \nabla \cdot \bPhi^{r}_{k+1}\right). \label{eqn:32_eq6}
    \end{gather}

    We write the left-hand side of (\ref{eqn:32_eq6}) as
    \begin{align}\label{eqn:32_eq7}
        & \frac{1}{4\Delta t}\left(\left\|\bPhi^{r}_{k+1}\right\|^{2} + \left\|2\bPhi^{r}_{k+1}-\bPhi^{r}_{k}\right\|^{2}\right) - \frac{1}{4\Delta t}\left(\left\|\bPhi^{r}_{k}\right\|^{2} + \left\|2\bPhi^{r}_{k}-\bPhi^{r}_{k-1}\right\|^{2}\right) \nonumber \\
        &  + \frac{1}{4\Delta t}\left\|\bPhi^{r}_{k+1} - 2\bPhi^{r}_{k} + \bPhi^{r}_{k-1}\right\|^{2} + \nu\left\|\nabla \bPhi^{r}_{k+1}\right\|^{2}
    \end{align}
    by applying the algebraic identity (\ref{eqn:algid}). 

    We now use the Cauchy-Schwarz, Young, and Poincaré inequalities to bound the first two addends on the right-hand side of \eqref{eqn:32_eq6} as follows:
    \begin{align}
        & \left(\br_{k+1}, \bPhi^{r}_{k+1}\right) \leq\left\|\br_{k+1}\right\|\left\| \bPhi^{r}_{k+1}\right\| \leq C_{1}\left(\frac{1}{4 \varepsilon}\left\|\br_{k+1}\right\|^{2}+\varepsilon\left\|\nabla \bPhi^{r}_{k+1}\right\|^{2}\right),  \label{eqn:32_eq8}  \\
        & \nu\left(\nabla \bfeta_{k+1}, \nabla \bPhi^{r}_{k+1}\right) \leq \nu\left\|\nabla \bfeta_{k+1}\right\|\left\|\nabla \bPhi^{r}_{k+1}\right\| \leq C_{2}\left(\frac{1}{4 \varepsilon}\left\|\nabla \bfeta_{k+1}\right\|^{2}+\varepsilon\left\|\nabla \bPhi^{r}_{k+1}\right\|^{2}\right),
        \label{eqn:32_eq9}
    \end{align}
    with $\varepsilon > 0 $ to be chosen later.

    The non-linear terms on the right-hand side of \eqref{eqn:32_eq6} can be written as follows: 
    {\begin{align}\label{eqn:32_eq10}
        & b^{*}\left(\bu_{k+1}, \bu_{k+1}, \bPhi^{r}_{k+1}\right) - b^{*}\left(D_{N}^{r}\oburk{k+1}, \bu^{r}_{k+1}, \bPhi^{r}_{k+1}\right) \nonumber \\
        & = b^{*}\left(\bu_{k+1}, \bu_{k+1}, \bPhi^{r}_{k+1}\right) - b^{*}\left(D_{N}^{r}\oburk[]{k+1}, \bu_{k+1}, \bPhi^{r}_{k+1}\right) \nonumber \\
        & \quad + b^{*}\left(D_{N}^{r}\oburk[]{k+1}, \bu_{k+1}, \bPhi^{r}_{k+1}\right) - b^{*}\left(D_{N}^{r}\oburk{k+1}, \bu^{r}_{k+1}, \bPhi^{r}_{k+1}\right) \nonumber \\
        & = b^{*}\left(\bu_{k+1}-D_{N}^{r}\oburk[]{k+1}, \bu_{k+1}, \bPhi^{r}_{k+1}\right) + b^{*}\left(D_{N}^{r}\oburk[]{k+1}, \bu_{k+1}, \bPhi^{r}_{k+1}\right) \\
        & \quad - b^{*}\left(D_{N}^{r}\oburk{k+1}, \bu_{k+1}, \bPhi^{r}_{k+1}\right) + b^{*}\left(D_{N}^{r}\oburk{k+1}, \bu_{k+1}, \bPhi^{r}_{k+1}\right) - b^{*}\left(D_{N}^{r}\oburk{k+1}, \bu^{r}_{k+1}, \bPhi^{r}_{k+1}\right) \nonumber \\
        & = b^{*}\left(\bu_{k+1}-D_{N}^{r}\oburk[]{k+1}, \bu_{k+1}, \bPhi^{r}_{k+1}\right) + b^{*}\left(D_{N}^{r}{\obfetak{k+1}}, \bu_{k+1}, \bPhi^{r}_{k+1}\right) \nonumber\\
        & \quad - b^{*}\left(D_{N}^{r}{\obPhirk{k+1}}, \bu_{k+1}, \bPhi^{r}_{k+1}\right) + b^{*}\left(D_{N}^{r}\oburk{k+1}, \bfeta_{k+1}, \bPhi^{r}_{k+1}\right) - \cancel{b^{*}\left(D_{N}^{r}\oburk{k+1}, \bPhi^{r}_{k+1}, \bPhi^{r}_{k+1}\right)}, \nonumber
    \end{align}}
    where in the last term we have used Lemma \ref{Lemma:skew_Symmetric}.

    For the last three addends in \eqref{eqn:32_eq10}, we apply Lemmas \ref{Lemma:skew_Symmetric} and \ref{thm:ROM211_split2}, and Young inequality: 
    {\begin{align}
        b^{*}\left(D_{N}^{r}\oburk{k+1}, \bfeta_{k+1}, \bPhi^{r}_{k+1}\right) & \stackrel{\ref{Lemma:skew_Symmetric}}{\leq}  C\left\|D_{N}^{r}\oburk{k+1}\right\|^{\frac{1}{2}}\left\|\nabla D_{N}^{r}\oburk{k+1}\right\|^{\frac{1}{2}}\left\|\nabla \bfeta_{k+1}\right\|\left\|\nabla \bPhi^{r}_{k+1}\right\| \nonumber \\
        &  \stackrel{\ref{thm:ROM211_split2}}{\leq}  C\left\|\bu^{r}_{k+1}\right\|^{\frac{1}{2}}\left\|\nabla \bu^{r}_{k+1}\right\|^{\frac{1}{2}}\left\|\nabla \bfeta_{k+1}\right\|\left\|\nabla \bPhi^{r}_{k+1}\right\| \nonumber \\
        &  \leq  C_{3}\left(\frac{1}{4 \varepsilon}\left\|\bu^{r}_{k+1}\right\|\left\|\nabla \bu^{r}_{k+1}\right\|\left\|\nabla \bfeta_{k+1}\right\|^{2}+\varepsilon\left\|\nabla \bPhi^{r}_{k+1}\right\|^{2}\right).  \label{eqn:32_eq11} \\
        b^{*}\left(D_{N}^{r}\obfetak{k+1}, \bu_{k+1}, \bPhi^{r}_{k+1}\right) &  \stackrel{\ref{Lemma:skew_Symmetric}}{\leq}   C\left\|D_{N}^{r}\obfetak{k+1}\right\|^{\frac{1}{2}}\left\|\nabla D_{N}^{r}\obfetak{k+1}\right\|^{\frac{1}{2}}\left\|\nabla \bu_{k+1}\right\|\left\|\nabla \bPhi^{r}_{k+1}\right\| \nonumber \\
        &  \stackrel{\ref{thm:ROM211_split2}}{\leq}  C\left\|\bfeta_{k+1}\right\|^{\frac{1}{2}} \left\|S^{r}\right\|_2^{\frac{1}{4}}\left\|\bfeta_{k+1}\right\|^{\frac{1}{2}}\left\|\nabla \bu_{k+1}\right\|\left\|\nabla \bPhi^{r}_{k+1}\right\| \nonumber \\
        &  \leq  C_{4}\left(\frac{1}{4 \varepsilon}\left\|S^{r}\right\|_2^{\frac{1}{2}}\left\|\bfeta_{k+1}\right\|^{2}\left\|\nabla \bu_{k+1}\right\|^{2} + \varepsilon\left\|\nabla \bPhi^{r}_{k+1}\right\|^{2}\right).   \label{eqn:32_eq12} \\
        b^{*}\left(D_{N}^{r}{\obPhirk{k+1}}, \bu_{k+1}, \bPhi^{r}_{k+1}\right) &  \stackrel{\ref{Lemma:skew_Symmetric}}{\leq}   C\left\|D_{N}^{r}{\obPhirk{k+1}}\right\|^{\frac{1}{2}}\left\|\nabla D_{N}^{r}{\obPhirk{k+1}}\right\|^{\frac{1}{2}}\left\|\nabla \bu_{k+1}\right\|\left\|\nabla \bPhi^{r}_{k+1}\right\| \nonumber \\
        &  \stackrel{\ref{thm:ROM211_split2}}{\leq}  C\left\|\bPhi^{r}_{k+1}\right\|^{\frac{1}{2}}\left\|\nabla \bu_{k+1}\right\|\left\|\nabla \bPhi^{r}_{k+1}\right\|^{\frac{3}{2}} \nonumber \\
        &  \leq  C_{5}\left(\frac{1}{4 \varepsilon}\left\|\nabla \bu_{k+1}\right\|^{4}\left\|\bPhi^{r}_{k+1}\right\|^{2}+\varepsilon\left\|\nabla \bPhi^{r}_{k+1}\right\|^{2}\right).   \label{eqn:32_eq13}
    \end{align}}
    Similarly, for the first addend in \eqref{eqn:32_eq10}, we get
    {\begin{align}
        b^{*}\left(\bu_{k+1}-D_{N}^{r}\oburk[]{k+1}, \bu_{k+1}, \bPhi^{r}_{k+1}\right) &  \stackrel{\ref{Lemma:skew_Symmetric}}{\leq}  C\left\|\bu_{k+1}-D_{N}^{r}\oburk[]{k+1}\right\|^{\frac{1}{2}}\left\|\nabla\left(\bu_{k+1}-D_{N}^{r}\oburk[]{k+1}\right)\right\|^{\frac{1}{2}}\left\|\nabla \bu_{k+1}\right\|\left\|\nabla \bPhi^{r}_{k+1}\right\| \nonumber \\
        & \leq C_{6}\left(\frac{1}{4 \varepsilon}\left\|\bu_{k+1}-D_{N}^{r}\oburk[]{k+1}\right\|\left\|\nabla\left(\bu_{k+1}-D_{N}^{r}\oburk[]{k+1}\right)\right\|\left\|\nabla \bu_{k+1}\right\|^{2} \right. \nonumber \\
        & \left. \quad \quad \quad + \varepsilon\left\|\nabla \bPhi^{r}_{k+1}\right\|^{2}\right). \label{eqn:32_eq14}
    \end{align}}

    Since $\bPhi^{r}_{k+1} \in \bX^{r}$ and $\bX^{r}$ is a space of weakly divergence-free functions, the pressure term on the right-hand side of (\ref{eqn:32_eq6}) can be written as
    \begin{equation*}
        -\left(p_{k+1}, \nabla \cdot \bPhi^{r}_{k+1}\right)=-\left(p_{k+1}-q^{h}, \nabla \cdot \bPhi^{r}_{k+1}\right),
    \end{equation*}
    where $q^{h}$ is an arbitrary function in $Q^{h}$. Thus, the pressure term can be estimated as follows by using the Cauchy-Schwarz and Young inequalities: 
    \begin{equation}\label{eqn:32_eq15}
        -\left(p_{k+1}, \nabla \cdot \bPhi^{r}_{k+1}\right) \leq C_{7} \left(\frac{1}{4 \varepsilon}\left\|p_{k+1}-q^{h}\right\|^{2}+\varepsilon\left\|\nabla\bPhi^{r}_{k+1}\right\|^{2}\right).
    \end{equation}

    Now choose $\varepsilon$ such that $4 \nu - 4 \varepsilon\left(C_1 + C_2 + C_{3} + C_{4} + C_{5} + C_{6} + C_7\right) = C_{8}$ for some $C_{8} \in (0, 4 \nu)$. 
    Then, substituting inequalities (\ref{eqn:32_eq7})-(\ref{eqn:32_eq9}), (\ref{eqn:32_eq11})-(\ref{eqn:32_eq14}), and (\ref{eqn:32_eq15}) in (\ref{eqn:32_eq6}), we obtain
    {\begin{align}\label{eqn:32_eq16}
        & \left\|\bPhi^{r}_{k+1}\right\|^{2} + \left\|2\bPhi^{r}_{k+1}-\bPhi^{r}_{k}\right\|^{2}\ - \left(\left\|\bPhi^{r}_{k}\right\|^{2} + \left\|2\bPhi^{r}_{k}-\bPhi^{r}_{k-1}\right\|^{2}\right) + \left\|\bPhi^{r}_{k+1} - 2\bPhi^{r}_{k} + \bPhi^{r}_{k-1}\right\|^{2} \nonumber \\
        & + C_{8} \Delta t \left\|\nabla \bPhi^{r}_{k+1}\right\|^{2} \nonumber \\
        & \quad \leq C_{8}\left(\Delta t\left\|\br_{k+1}\right\|^{2} + \Delta t\left\|\nabla \bfeta_{k+1}\right\|^{2 } + \Delta t\left\|\bu^{r}_{k+1}\right\|\left\|\nabla \bu^{r}_{k+1}\right\|\left\|\nabla \bfeta_{k+1}\right\|^{2}\right. \nonumber \\
        & \quad \quad + \Delta t \left\|S^{r}\right\|_2^{\frac{1}{2}}\left\|\bfeta_{k+1}\right\|^{2} \left\|\nabla \bu_{k+1}\right\|^{2} + \Delta t\left\|\nabla \bu_{k+1}\right\|^{4}\left\|\bPhi^{r}_{k+1}\right\|^{2} \nonumber \\
        & \left.\quad \quad + \Delta t \left\|\bu_{k+1} - D_{N}^{r}\oburk[]{k+1}\right\|\left\|\nabla\left(\bu_{k+1}-D_{N}^{r}\oburk[]{k+1}\right)\right\|\left\|\nabla \bu_{k+1}\right\|^{2} + \Delta t\left\|p_{k+1}-q_{h}\right\|^{2}\right).
    \end{align}}

    Summing (\ref{eqn:32_eq16}) from $k=1$ to $k=K-1$, we have
    {\begin{align}\label{eqn:32_eq17} 
        & \left\|\bPhi^{r}_{K}\right\|^{2} + 
        C_{8} \Delta t \sum_{k=1}^{K-1}\left\|\nabla \bPhi^{r}_{k+1}\right\|^{2}
         \leq \left\|\bPhi^{r}_{1}\right\|^{2} + \left\|2\bPhi^{r}_{1} - \bPhi^{r}_{0}\right\|^{2} \nonumber \\
        & \quad + C_{8} \Delta t \left(\sum_{k=1}^{K-1}\left\|\br_{k+1}\right\|^{2} + \sum_{k=1}^{K-1}\left\|\nabla \bfeta_{k+1}\right\|^{2} + \sum_{k=1}^{K-1}\left\|\bu^{r}_{k+1}\right\|\left\|\nabla \bu^{r}_{k+1}\right\|\left\|\nabla \bfeta_{k+1}\right\|^{2}\right. \nonumber \\
        & \quad +\sum_{k=1}^{K-1}\left\|\bu_{k+1}-D_{N}^{r}\oburk[]{k+1}\right\|\left\|\nabla\left(\bu_{k+1}-D_{N}^{r}\oburk[]{k+1}\right)\right\|\left\|\nabla \bu_{k+1}\right\|^{2}  \nonumber \\
        & \left. \quad  + \left\|S^{r}\right\|_2^{\frac{1}{2}} \sum_{k=1}^{K-1}\left\|\nabla \bu_{k+1}\right\|^{2}\left\|\bfeta_{k+1}\right\|^{2} + \sum_{k=1}^{K-1}\left\|p_{k+1}-q_{h}\right\|^{2} + \sum_{k=1}^{K-1}\left\|\nabla \bu_{k+1}\right\|^{4}\left\|\bPhi^{r}_{k+1}\right\|^{2} \right).
    \end{align}}

    The first two terms on the right-hand side of (\ref{eqn:32_eq17}) vanish, since $\bu^{r}_{0}=\bw^{r}_{0}$ and $\bu^{r}_{1}=\bw^{r}_{1}$.

    To estimate the third term on the right-hand side of~\eqref{eqn:32_eq17}, we expand $\bu^k$ and $\bu^{k-1}$ about $t_{k+1}$ with Taylor's theorem. Then,
    
    \begin{align}
        \bu^k = \bu^{k+1} - \bu_t^{k+1} \Delta t + \frac{1}{2} \bu_{tt}^{k+1} \Delta t^2 
        -
        \int_{t_k}^{t_{k+1}} \frac{1}{2} \bu_{ttt}(s) (
        t_{k} - s)^2 ds, \label{eq:taylor_exp_u_n}\\
          \bu^{k-1} = \bu^{k+1} - 2\bu_t^{k+1} \Delta t + 2 \bu_{tt}^{k+1} \Delta t^2 
          -
          \int_{t_{k-1}}^{t_{k+1}} \frac{1}{2} \bu_{ttt}(s) (
          t_{k-1} - s)^2 ds. 
          \label{eq:taylor_exp_u_n-1}
    \end{align}
    Multiplying 
    \eqref{eq:taylor_exp_u_n} by $-4$, adding the result  to~\eqref{eq:taylor_exp_u_n-1}, and rearranging terms, we obtain 
    \begin{align}
        2 \bu_t^{k+1} \Delta t - \left(3 \bu^{k+1} - 4 \bu^k + \bu^{k-1} \right)  &= -2 \int_{t_k}^{t_{k+1}} \bu_{ttt}(s) (
        t_{k} - s)^2 ds \notag\\
        &\phantom{=} + \int_{t_{k-1}}^{t_{k+1}} \frac{1}{2} \bu_{ttt}(s) (
        t_{k-1} - s)^2 ds. \label{eq:taylor_expand_r_k+1}
    \end{align}
    The left-hand side of~\eqref{eq:taylor_expand_r_k+1} divided by $2 \Delta t$ is $\br_{k+1}$.
    When taking the $L^2$ norm in space of~\eqref{eq:taylor_expand_r_k+1}, the result is the norm of a difference of two 
    similar integrals. We use the triangle inequality to bound these norms individually such that both have a factor of $\Delta t^{3/2}$ multiplied by a constant.
    We now demonstrate how to bound the first term on the right-hand side of~\eqref{eq:taylor_expand_r_k+1}: 
    \begin{equation}
    \label{eq:taylor_exp_intermediate}
        \left\| -\frac{1}{2\Delta t} \cdot 2 \int_{t_k}^{t_{k+1}} \bu_{ttt}(s) (
        t_{k} - s)^2 ds \right\|
        = \left[ \int_\Omega \left(\int_{t_{k}}^{t_{k+1}} \bu_{ttt}(s) \frac{(t_{k} - s)^2}{\Delta t}  ds\right)^2 d\bx \right]^{1/2}.
    \end{equation}
    Using the Cauchy-Schwarz inequality 
    for the time integral yields
    \begin{flalign}
        &\left[ \int_\Omega \left(\int_{t_{k}}^{t_{k+1}} \bu_{ttt}(s) \frac{(t_{k} - s)^2}{\Delta t}  ds\right)^2  d\bx \right]^{1/2} 
        \notag \\ &\qquad\leq \left[ \int_\Omega \left(\int_{t_{k}}^{t_{k+1}} \bigl(\bu_{ttt}(s)  \bigr)^2 ds\right) \left(\int_{t_{k}}^{t_{k+1}} \frac{(t_{k} - s)^4}{\Delta t^2} ds \right) d\bx \right]^{1/2}. 
\label{eq:taylor_exp_CS_step}
    \end{flalign}
    Noting that $|t_{k} - s| \leq  \Delta t$ 
    and evaluating the integrals in~\eqref{eq:taylor_exp_CS_step}, we 
    obtain 
    \begin{equation}
         \left[ \int_\Omega \left(\int_{t_{k}}^{t_{k+1}} \bigl(\bu_{ttt}(s)  \bigr)^2 ds\right) \left(\int_{t_{k}}^{t_{k+1}} \frac{(t_{k} - s)^4}{\Delta t^2} ds \right) d\bx \right]^{1/2} \leq  \| \bu_{ttt} \|_{L^2(L^2)} \Delta t^{3/2},
        \label{eq:taylor_exp_final_result}
    \end{equation}
    where $ \| \cdot \|_{L^2(L^2)}$ refers to the $L^2$ norm in both space and time over $\Omega$ and the time interval $[t_{k-1},t_{k+1}]$. This change to a larger time interval than in~\eqref{eq:taylor_exp_CS_step} is informed by the next point, which is that the second integral in~\eqref{eq:taylor_expand_r_k+1} may be bounded in a similar manner by noting that $|t_{k-1}-s| \leq 2 \Delta t$ over $[t_{k-1},t_{k+1}]$. Using
    \eqref{eq:taylor_expand_r_k+1}-\eqref{eq:taylor_exp_final_result} with a constant $C_{9}$ depending on $\| \bu_{ttt} \|_{L^2(L^2)}$, we 
    obtain the following bound:

    \begin{equation}
    \label{eqn:32_eq18}
            \Delta t \sum_{k=1}^{K-1}\left\|\br_{k+1}\right\|^{2} \leq C_{9} \Delta t^4.
    \end{equation}

    Using~\eqref{eq:L2_ROM_projection_error_2}, the fourth term on the right-hand side of~\eqref{eqn:32_eq17} can be estimated as follows: 
    {\begin{equation}\label{eqn:32_eq19}
        \Delta t \sum_{k=1}^{K-1}\left\|\nabla \bfeta_{k+1}\right\|^{2} \leq C_{10}\left(h^{2m+2} + h^{2 m+4}\left\|S^{r}\right\|_{2}
         + \Delta t^4\left(1+\left\|S^{r}\right\|_{2}\right) + \sum_{j=r+1}^{R}\left\|\nabla\bvarphi_{j}\right\|^{2} \lambda_{j}\right).
    \end{equation}}

    To estimate the fifth term on the right-hand side of (\ref{eqn:32_eq17}), we use 
    Lemma~\ref{lemma:assumption-u-gradu-gradeta}:
    \begin{align}\label{eqn:32_eq20}
         & \Delta t \sum_{k=1}^{K-1}\left\|\bu^{r}_{k+1}\right\|\left\|\nabla \bu^{r}_{k+1}\right\|\left\|\nabla \bfeta_{k+1}\right\|^{2} 
         \nonumber \\
        & \leq 
        C_{11}\left(h^{2m+2} + h^{2m+4}\left\|S^{r}\right\|_{2} + \Delta t^4\left(1+\left\|S^{r}\right\|_{2}\right) + \sum_{j=r+1}^{R}\left\|\nabla\bvarphi_{j}\right\|^{2} \lambda_{j}\right).
    \end{align}

    To estimate the sixth term on the right-hand side of (\ref{eqn:32_eq17}), we use \eqref{eqn:Lem211_s2}, Lemma \ref{lemma:Layton_213_Rom}, and Assumption \ref{ass:reg_ex_sol}: 
    \begin{align}\label{eqn:32_eq21}
        & \Delta t \sum_{k=1}^{K-1}\left\|\bu_{k+1}-D_{N}^{r}\oburk{k+1}\right\|\left\|\nabla\left(\bu_{k+1}-D_{N}^{r}\oburk{k+1}\right)\right\|\left\|\nabla \bu_{k+1}\right\|^{2} \nonumber \\
        & \quad \stackrel{\eqref{eqn:Lem211_s2}}{\leq} \hat{C}_{12} \sqrt{\|S^{r}\|_2} \ \Delta t \sum_{k=1}^{K-1}\left\|\bu_{k+1}-D_{N}^{r}\oburk{k+1}\right\|^2 \left\|\nabla \bu_{k+1}\right\|^{2} \nonumber \\
        & \quad \stackrel{\ref{lemma:Layton_213_Rom}}{\leq} \widetilde{C}_{12} \sqrt{\|S^{r}\|_2} \ \Delta t \sum_{k=1}^{K-1}\left\|\nabla \bu_{k+1}\right\|^{2} \left[\left(h^{2m+4}+\Delta t^4+\sum_{j=r+1}^{R} \lambda_{j}\right)\right. \nonumber \\
        & \quad \quad \left.+\delta^{2}\left(h^{2m+2}+h^{2m+4}\left\|S^{r}\right\|_{2}+\Delta t^4\left(1+\left\|S^{r}\right\|_{2}\right) +\sum_{j=r+1}^{R}\left\|\nabla\bvarphi_{j}\right\|^{2} \lambda_{j}\right)+\delta^{4}\right] \nonumber \\
        & \quad \stackrel{\ref{ass:reg_ex_sol}}{\leq} C_{12} \sqrt{\|S^{r}\|_2} \ \left[\left(h^{2m+4}+\Delta t^4+\sum_{j=r+1}^{R} \lambda_{j}\right)\right. \nonumber \\
        & \quad \quad \left.+\delta^{2}\left(h^{2m+2} + h^{2m+4}\left\|S^{r}\right\|_{2}+\Delta t^4 \left(1+\left\|S^{r}\right\|_{2}\right) +\sum_{j=r+1}^{R}\left\|\nabla\bvarphi_{j}\right\|^{2} \lambda_{j}\right)+\delta^{4}\right].
    \end{align}

    Using Assumption \ref{assump:POD_approximation}, the seventh term on the right-hand side of (\ref{eqn:32_eq17}) can be bounded as follows: 
    \begin{align}\label{eqn:32_eq22}
        &\left\|S^{r}\right\|_2^{\frac{1}{2}} \Delta t \sum_{k=1}^{K-1}\left\|\nabla \bu_{k+1}\right\|^{2}\left\|\bfeta_{k+1}\right\|^{2} \nonumber \\
        & \quad \stackrel{(\ref{eq:POD_approximation_1})}{\leq} \widetilde{C}_{13} \Delta t \left\|S^{r}\right\|_2^{\frac{1}{2}}\sum_{k=1}^{K-1}\left\|\nabla \bu_{k+1}\right\|^{2}\left(h^{2m+4}+\Delta t^4+\sum_{j=r+1}^{R} \lambda_{j}\right) \nonumber \\
        & \quad \stackrel{\ref{ass:reg_ex_sol}}{\leq} C_{13}\left\|S^{r}\right\|_2^{\frac{1}{2}}\left(h^{2m+4}+\Delta t^4+\sum_{j=r+1}^{R} \lambda_{j}\right).
    \end{align}

    Since in (\ref{eqn:32_eq15}) $q^{h}$ is an arbitrary function in $Q^{h}$, we can use Assumption \ref{assumption:finite-element} to bound the eighth term on the right-hand side of (\ref{eqn:32_eq17}) as follows: 
    \begin{equation}\label{eqn:32_eq23}
        \Delta t \sum_{k=1}^{K-1}\left\|p_{k+1}-q_{h}\right\|^{2} \leq C_{14} h^{2 m + 2}.
    \end{equation}

    Collecting (\ref{eqn:32_eq18})-(\ref{eqn:32_eq23}), equation (\ref{eqn:32_eq17}) becomes
    {\begin{equation}\label{eqn:32_eq24}
        \begin{aligned}
            & \left\|\bPhi^{r}_{K}\right\|^{2} + C_{8} \Delta t \sum_{k=1}^{K-1} \left\|\nabla \bPhi^{r}_{k+1} \right\|^{2} \\
            & \quad \leq C_{15}\left\{\Delta t \sum_{k=1}^{K-1}\left\|\nabla\bu_{k+1}\right\|^{4}\left\|\bPhi^{r}_{k+1}\right\|^{2}\right. \\
            & \quad \quad + \Delta t^4 + \left(h^{2m+2} + h^{2m+4}\left\|S^{r}\right\|_{2}
             + \Delta t^4\left(1+\left\|S^{r}\right\|_{2}\right) + \sum_{j=r+1}^{R}\left\|\nabla\bvarphi_{j}\right\|^{2} \lambda_{j}\right) 
             \\
            & \quad \quad + \left(h^{2m+2} + h^{2m+4}\left\|S^{r}\right\|_{2} + \Delta t^4\left(1+\left\|S^{r}\right\|_{2}\right) + \sum_{j=r+1}^{R}\left\|\nabla\bvarphi_{j}\right\|^{2} \lambda_{j}\right) + \left\|S^{r}\right\|_2^{\frac{1}{2}} \left(h^{2m+4}+\Delta t^4+\sum_{j=r+1}^{R} \lambda_{j}\right)\\
            & \quad \quad  +\delta^{2} \left\|S^{r}\right\|_2^{\frac{1}{2}} \left(h^{2m+2}+h^{2m+4}\left\|S^{r}\right\|_{2}+\Delta t^4 \left(1+\left\|S^{r}\right\|_{2}\right) +\sum_{j=r+1}^{R}\left\|\nabla\bvarphi_{j}\right\|^{2} \lambda_{j}\right)\\
            & \quad \quad + \left.\delta^{4} \left\|S^{r}\right\|_2^{\frac{1}{2}}  + \left\|S^{r}\right\|_2^{\frac{1}{2}}\left(h^{2m+4}+\Delta t^4+\sum_{j=r+1}^{R} \lambda_{j}\right) + h^{2m+2}\right\} 
            \\ 
            & \quad = C_{15}\left\{\Delta t \sum_{k=1}^{K-1}\left\|\nabla\bu_{k+1}\right\|^{4}\left\|\bPhi^{r}_{k+1}\right\|^{2}\right. \\
            & \quad \quad + \left(h^{2m+2} + h^{2m+4}\left\|S^{r}\right\|_{2} 
             + \Delta t^4\left(1+\left\|S^{r}\right\|_{2}\right) + \sum_{j=r+1}^{R}\left\|\nabla\bvarphi_{j}\right\|^{2} \lambda_{j}\right) \\
            & \quad \quad  +\delta^{2} \left\|S^{r}\right\|_2^{\frac{1}{2}} \left(h^{2m+2} + h^{2m+4}\left\|S^{r}\right\|_{2}+\Delta t^4 \left(1+\left\|S^{r}\right\|_{2}\right) +\sum_{j=r+1}^{R}\left\|\nabla\bvarphi_{j}\right\|^{2} \lambda_{j}\right)\\
            & \quad \quad +\left.\delta^{4} \left\|S^{r}\right\|_2^{\frac{1}{2}} + \left\|S^{r}\right\|_2^{\frac{1}{2}} \left(h^{2m+4}+\Delta t^4+{\sum_{j=r+1}^{R} \lambda_{j}}\right)\right\}\\
            & \quad \stackrel{(\ref{eqn:f_32})}{=} C_{15}\left\{\Delta t \sum_{k=1}^{K-1}\left\|\nabla \bu_{k+1}\right\|^{4}\left\|\bPhi^{r}_{k+1}\right\|^{2} + \ \mathcal{G}(\delta, \Delta t, h, \|S^{r}\|_2,\{\lambda_j\}_{j=r+1}^{R},\{\|\nabla\bvarphi_j\|\}_{j=r+1}^{R})\right\}.
        \end{aligned}
    \end{equation}}

    Lemma \ref{lemma:discrete_gronwell} implies that, for small enough $\Delta t$ (i.e., $\Delta t<$ $(C_{15} \max _{1 \leq k \leq K}\left\|\nabla \bu_{k}\right\|^{4})^{-1})$, the following inequality holds: 
    \begin{align}\label{eqn:32_eq25}
        & \left\|\bPhi^{r}_{K}\right\|^{2} + C_{8} \Delta t \sum_{k=1}^{K-1}\left\|\nabla \bPhi^{r}_{k+1}\right\|^{2} \nonumber \\
        & \quad \leq C_{16} \ \mathcal{G}(\delta, \Delta t, h, \|S^{r}\|_2,\{\lambda_j\}_{j=r+1}^{R},\{\|\nabla\bvarphi_j\|\}_{j=r+1}^{R}).
    \end{align}

    By using (\ref{eqn:32_eq25}), the triangle inequality, and (\ref{eq:L2_ROM_projection_error_2})-(\ref{eq:POD_approximation_1}),
    we obtain
    \begin{align}\label{eqn:32_eq26}
        & \left\|\bu_{K}-\bu^{r}_{K}\right\|^{2}+\Delta t \sum_{k=1}^{K-1}\left\|\nabla\left(\bu_{k+1}-\bu^{r}_{k+1}\right)\right\|^{2} \\
        & \quad \leq C \mathcal{G}(\delta, \Delta t, h, \|S^{r}\|_2,\{\lambda_j\}_{j=r+1}^{R},\{\|\nabla\bvarphi_j\|\}_{j=r+1}^{R}).
    \end{align}

This completes the proof.
\unskip\nobreak\hspace*{9.5cm} \end{proof}

\section{Numerical Results}
    \label{sec:numerical-results}
    
In this section, we conduct a numerical investigation of the theoretical results presented in Section \ref{sec:adl-rom-numerical-analysis}. Our aim is to assess if the numerical results align with the {\emph{a priori}} error 
bounds for 
ROM approximate deconvolution in Lemma \ref{lemma:Layton_213_Rom} and ADL-ROM approximation in Theorem \ref{thm:32}.

For our numerical investigation, we adopt the same test problem and computational setup as those described in \citep[Section 5]{xie2018numerical}. The problem involves the $2$D incompressible 
NSE with a known analytical solution. Specifically, the exact velocity, denoted by $\bu=(u, v)$, is defined by $u=\frac{2}{\pi} \arctan (-500(y-t)) \sin (\pi y)$ and $v=\frac{2}{\pi} \arctan (-500(x-t)) \sin (\pi x)$. Additionally, the exact pressure is constant, given by $p=0$. We set the viscosity $\nu=10^{-3}$ and select the forcing term to match the exact solution. The spatial domain $[0,1] \times [0,1]$ is discretized using Taylor-Hood finite elements with a mesh size of $h=\frac{1}{64}$.

To construct the 
POD basis, we gather snapshots over the time interval $[0,1]$ at intervals of $\Delta t=10^{-2}$, interpolating the analytical solution $(u, v)$ on the finite element mesh. Rather than employing the common centering trajectory approach, we directly apply the method of snapshots to the snapshot data. The resulting POD basis has a dimension of $100$.

\begin{remark}[Boundary Conditions]
    The exact solution satisfies the following boundary conditions:
    \begin{equation}
        \begin{cases}
            u = 0 \quad & \text{on } [0,1]\times\{0\} \text{ and } [0,1]\times\{1\}, \\
            v = 0 \quad & \text{on } \{0\}\times[0,1] \text{ and } \{1\}\times[0,1], \\
            \frac{\partial v}{\partial n} = 0 \quad & \text{on } [0,1]\times\{0\} \text{ and } [0,1]\times\{1\}, \\
            \frac{\partial u}{\partial n} = 0 \quad & \text{on } \{0\}\times[0,1] \text{ and } \{1\}\times[0,1].
    \end{cases}
    \label{eqn:remark:ROM_BCs}
    \end{equation}
Since the snapshots are obtained from the exact solution (by interpolating it on the finite element mesh), the snapshots inherit the homogeneous Dirichlet and Neumann boundary conditions~\eqref{eqn:remark:ROM_BCs} satisfied by the exact solution.
Thus, at the ROM level there is no need to enforce the boundary conditions since the POD basis and, as a result, the ROM solution will automatically satisfy the homogeneous Dirichlet and Neumann boundary conditions~\eqref{eqn:remark:ROM_BCs}.
    \label{remark:ROM_BCs}
\end{remark}

\subsection{AD Rates of Convergence}
\label{sec:AD_ROC}

\begin{figure}
\centering
\subcaptionbox{$r = 99$\label{fig:AD_scaling_r_99}}{\includegraphics[width=0.45\textwidth]{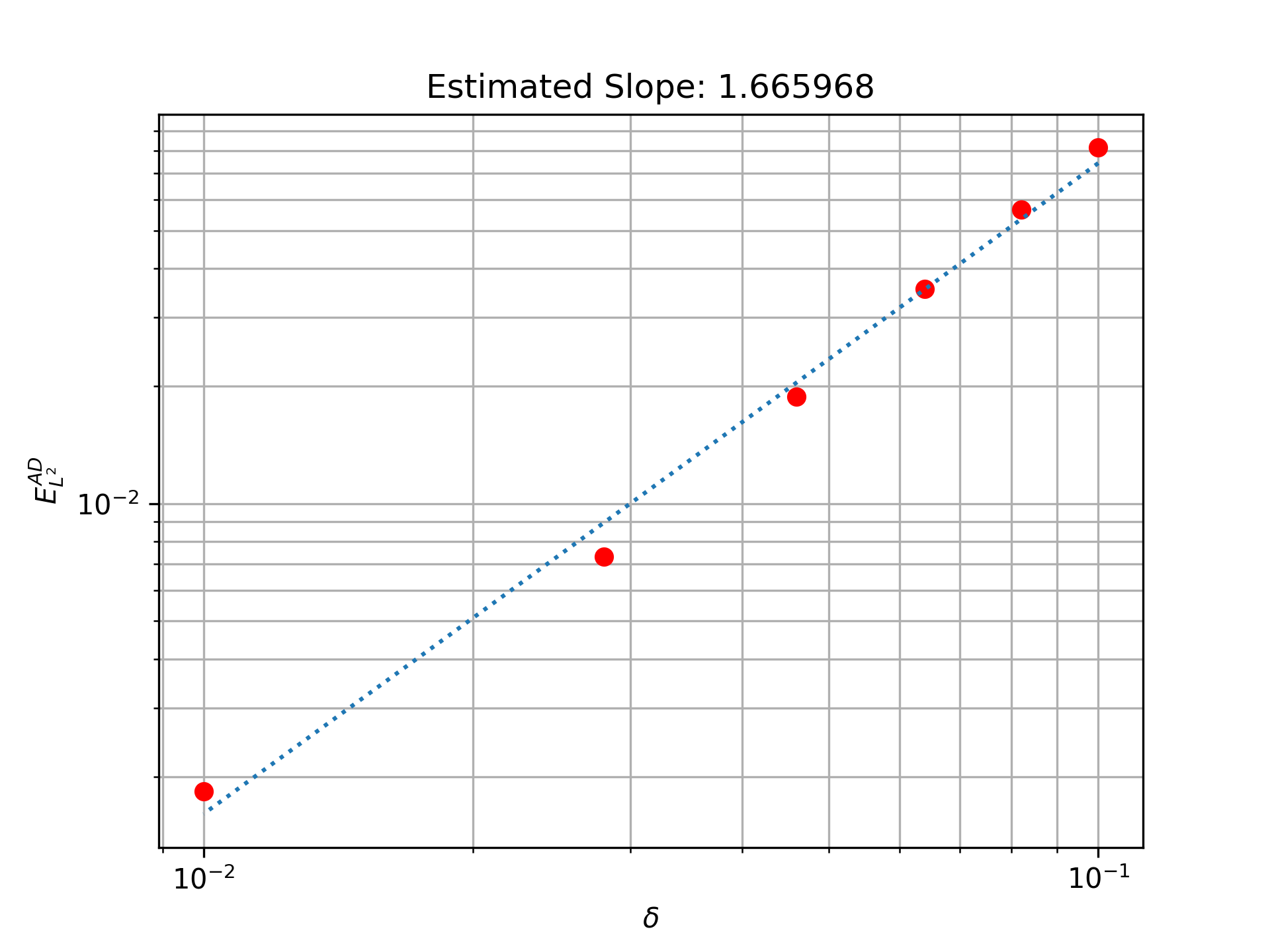}}%
\hfill
\subcaptionbox{$r=100$\label{fig:AD_scaling_r_100}}{\includegraphics[width=0.45\textwidth]{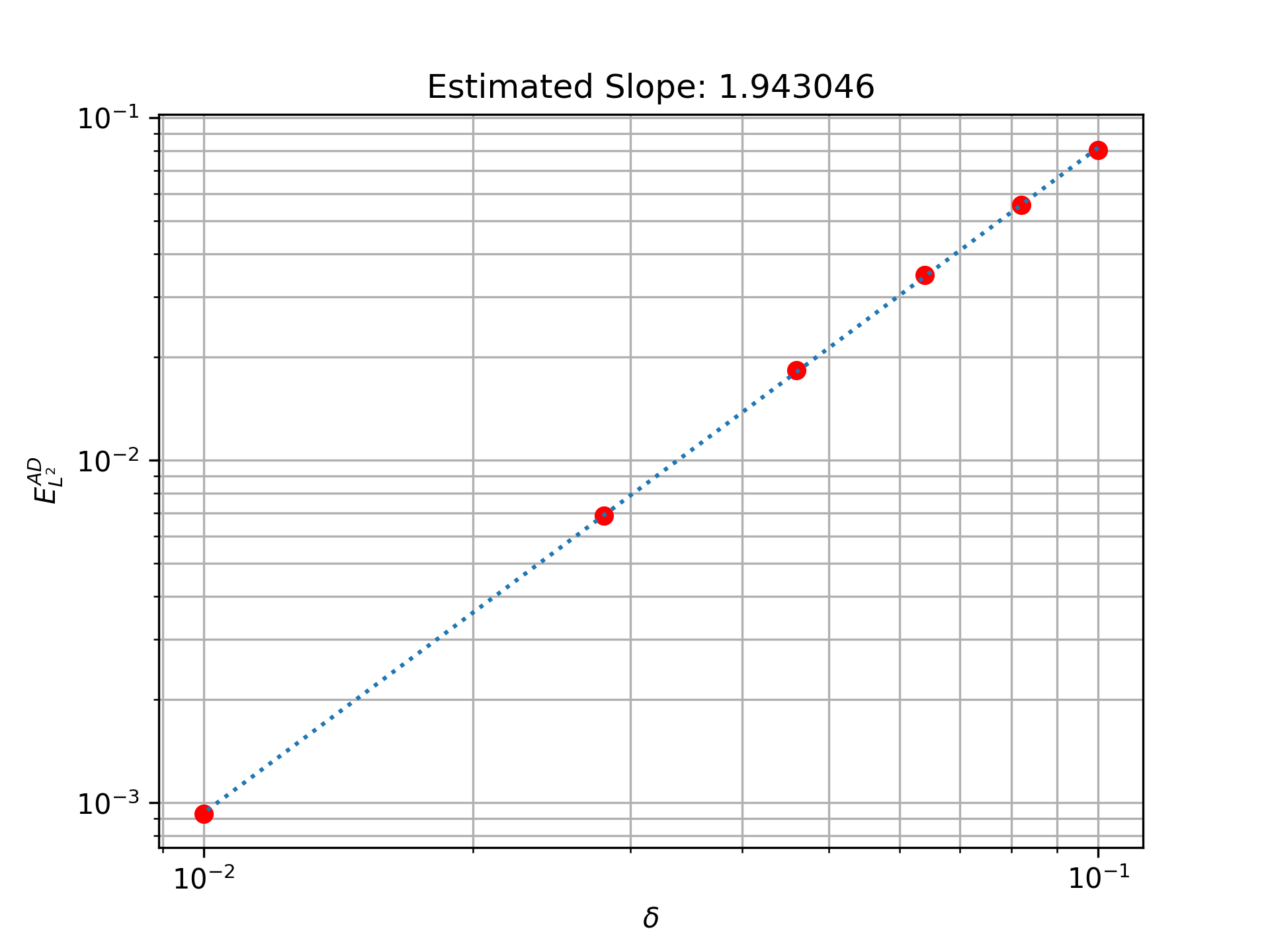}}%
\caption{Approximate deconvolution rates of convergence for (a) $r=99$ and (b) $r=100$.}
\label{fig:ad_scaling_results}
\end{figure}

In this section, we perform a numerical investigation of the van Cittert approximate deconvolution method itself, independently from the online running of the ROM.
Specifically, we investigate numerically how $\| \bu_k - D_N^r \obuk{k} \|$ scales with respect to filter 
radius, $\delta$.
To this end, we note that a theoretical scaling is provided by the {\emph{a priori}} error bound in Lemma \ref{lemma:Layton_213_Rom}, which, for completeness, we include below:
\begin{align}
     \| \bu_k - D_N^r \obuk{k} \|^2 &\leq C(N) \left(h^{2m+4} + \Delta t^4 + \sum_{j=r+1}^R \lambda_j\right) \nonumber \\
    &+ C(N) \delta^2 \left( h^{2m+2} + \| S^{r} \|_2 \ h^{2m+4} +  (1 + \|S^{r} \|_2) \Delta t^4 + \sum_{j=r+1}^R \| \nabla\bvarphi_j \|^2 \lambda_j \right) \nonumber \\
  &+ C(N) \delta^4 \| \Delta \bu \|^2.
  \label{eq:5.1_lemma}
\end{align}
To investigate how $\| \bu_k - D_N^r \obuk{k} \|$ scales with respect to the filter parameter $\delta$, we make the other sources of error (i.e., those associated with the finite element and ROM truncation errors) small compared to those associated with $\delta.$ 
Our primary sources of error in Lemma \ref{lemma:Layton_213_Rom} depend on $h$, $\Delta t$, $\delta$, and the ROM truncation errors $\Lambda_{L^{2}}^{r}=\sum_{j=r+1}^R \lambda_j$ and $\Lambda_{H^{1}}^{r}=\sum_{j = r+1}^R \| \nabla{\bvarphi_j} \|^2 \lambda_j$. Because our snapshots 
are projections of the analytical solution 
on the finite element space, the approximation error is independent of the time step, $\Delta t$. Fixing $h = 1/64$, $m = 2$ (
since we use the Taylor-Hood elements), and $r = 99$, we have that $h^{2m} = \mathcal{O}(10^{-8})$, $\Lambda_{L^{2}}^{r} = \mathcal{O}(10^{-5})$, $\Lambda_{H^{1}}^{r} = \mathcal{O}(10^{0})$, 
and $\| S_r \|_2 = \mathcal{O}(10^5)$. 

The {\emph{a priori}} error bound~\eqref{eq:5.1_lemma} suggests the scaling $\| \bu_k - D_N^r \obuk{k} \| = \mathcal{O}(\delta^2)$ when all other sources of error are small.
To investigate numerically this theoretical scaling, we use the following error metric, with $\bu_K^h$ referring to the FE projection of the known solution at the final snapshot used for POD: 
\begin{equation}
    E_{L^2}^{AD} = \| \bu_K^h - D_N^r \oburk[h]{K} \|.
\end{equation}
In Fig.~\ref{fig:ad_scaling_results}, we plot the error $E_{L^2}^{AD}$ for $r=99$ and $r=100$, the filter radius $\delta$ taken to be $6$ evenly spaced 
values between $10^{-1}$ and $10^{-2}$, and the van Cittert parameter $N = 5$. 

The theoretical scaling $E_{L^2}^{AD} = \mathcal{O}(\delta^2)$ is recovered numerically for $r=100$ (Fig.~\ref{fig:AD_scaling_r_100}). 
We note, however, that the numerical scaling for $r=99$ in Fig.~\ref{fig:AD_scaling_r_99} is somewhat lower than the expected theoretical scaling.
This lower scaling is due to the relatively large size of the term $\Lambda_{H^{1}}^{r}$ on the RHS of~\eqref{eq:5.1_lemma}.

\subsection{ADL-ROM Rates of Convergence}
\label{sec:ADL-ROM_ROC}

\begin{figure}
\centering
\begin{minipage}[b]{.37\textwidth}
  \centering
  {\renewcommand\arraystretch{1.5}
    \begin{tabular}{|c|c|}
        \hline
        $\delta$ & $E_{L^{2}}^{ADL-ROM}$ \\
        \hline
        $6.50\times10^{-1}$ & $2.66\times10^{-1}$ \\
        $5.00\times10^{-1}$ & $8.72\times10^{-2}$ \\
        $2.50\times10^{-1}$ & $4.74\times10^{-2}$ \\
        $1.80\times10^{-1}$ & $2.44\times10^{-2}$ \\
        $1.25\times10^{-1}$ & $8.05\times10^{-3}$ \\
        $6.25\times10^{-2}$ & $2.06\times10^{-3}$ \\
        \hline
    \end{tabular}}\\[1.5em]
    \captionof{table}{ADL-ROM approximation error $E_{L^{2}}^{ADL-ROM}$ for decreasing $\delta$ values.}
    \label{tbl:Delta_Scaling_ROM}
\end{minipage}
\hspace{1cm}
\begin{minipage}[b]{.45\textwidth}
    \centering
    \includegraphics[width=\textwidth]{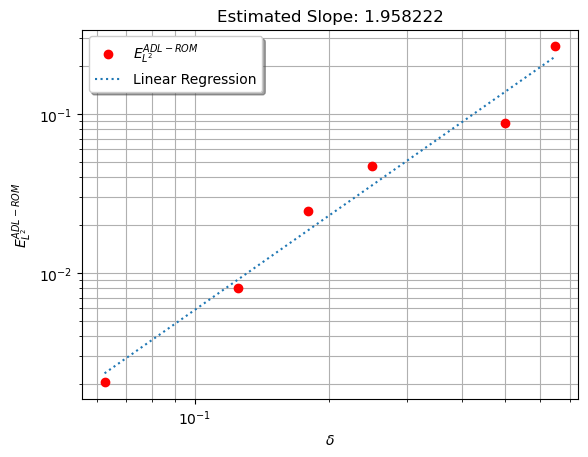}
    \caption{Linear regression of $E_{L^{2}}^{ADL-ROM}$ with respect to $\delta$.\\}
    \label{fig:Delta_Scaling_ROM}
\end{minipage}\hfill
\end{figure}

\begin{figure}
\centering
\begin{minipage}[b]{.40\textwidth}
  \centering
  {\renewcommand\arraystretch{1.5}
    \begin{tabular}{|c|c|c|}
        \hline
        $r$ & $\Lambda^{r}_{H^1}$ & $E_{L^{2}}^{ADL-ROM}$ \\
        \hline
        10 & $2.14\times10^{4}$ & $9.55\times10^{-2}$ \\
        20 & $1.71\times10^{4}$ & $4.98\times10^{-2}$ \\
        30 & $1.34\times10^{4}$ & $3.37\times10^{-2}$ \\
        40 & $1.02\times10^{4}$ & $2.34\times10^{-2}$ \\
        50 & $7.36\times10^{3}$ & $1.83\times10^{-2}$ \\
        \hline
    \end{tabular}}\\[2.5em]
    \captionof{table}{ADL-ROM approximation error $E_{L^{2}}^{ADL-ROM}$ for increasing $r$ values.}
    \label{tbl:r_Scaling_ROM}
\end{minipage}
\hspace{1cm}
\begin{minipage}[b]{.45\textwidth}
    \centering
    \includegraphics[width=\textwidth]{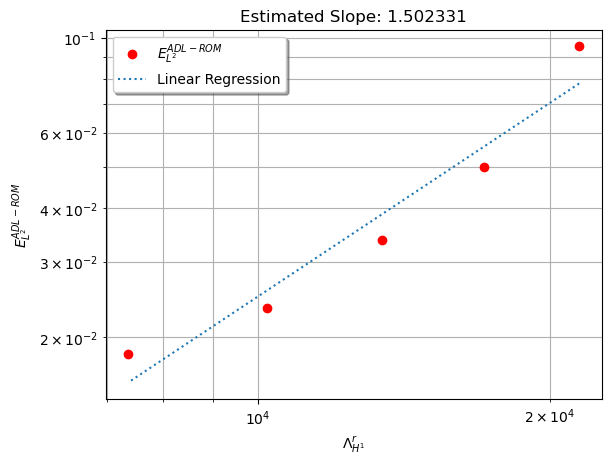}
    \caption{Linear regression of $E_{L^{2}}^{ADL-ROM}$ with respect to $\Lambda^{r}_{H^1}$.}
    \label{fig:r_Scaling_ROM}
\end{minipage}\hfill
\end{figure}

In this section, we perform a numerical investigation of the ADL-ROM 
{\emph{a priori}} error bound \eqref{eqn:theorem-bound}--(\ref{eqn:f_32}) in Theorem \ref{thm:32}. The ADL-ROM approximation error at the final time step is $E_{L^{2}}^{ADL-ROM}=\left\|\bu^{h}_{K} - \bu_{K}^r \right\|$, where $\bu_{K}^r \ $
 is the ADL-ROM solution at the final time step. We numerically investigate the rates of convergence of $E_{L^{2}}^{ADL-ROM}$ with respect to the filter radius $\delta$ and the ROM truncation error $\Lambda_{H^{1}}^{r}=\sum_{j=r+1}^{R}\|\nabla\bvarphi_j\|^2\lambda_j$.

To determine the ADL-ROM approximation error rate of convergence with respect to $\delta$, we fix $h=\frac{1}{64}$, $r=99$, $\Delta t=10^{-3}$, $N=5$, and vary $\delta$. With these choices, $h^{2 m}=\mathcal{O}\left(10^{-8}\right)$, $\Delta t^{4}=\mathcal{O}\left(10^{-12}\right)$, $\Lambda_{L^{2}}^{r}=\mathcal{O}\left(10^{-5}\right)$, $\Lambda_{H^{1}}^{r}=\mathcal{O}\left(10^{0}\right)$, and $\left\|S^{r}\right\|_{2}=\mathcal{O}\left(10^{5}\right)$.  Thus, the theoretical ADL-ROM 
{\emph{a priori}} error bound predicts the following rate of convergence of $E_{L^{2}}^{ADL-ROM}$ with respect to $\delta$:

\begin{equation}
E_{L^{2}}^{ADL-ROM}=\mathcal{O}\left(\delta^{2}\right).
\end{equation}

The ADL-ROM approximation error $E_{L^{2}}^{ADL-ROM}$ is listed in Table $\ref{tbl:Delta_Scaling_ROM}$ for decreasing $\delta$ values. The corresponding linear regression, which is shown in Fig.~\ref{fig:Delta_Scaling_ROM}, indicates the following ADL-ROM approximation error rate of convergence with respect to $\delta$:

\begin{equation}
E_{L^{2}}^{ADL-ROM}=\mathcal{O}\left(\delta^{1.96}\right).
\end{equation}
Thus, the theoretical rate of convergence is numerically recovered.

To determine the ADL-ROM approximation error rate of convergence with respect to $\Lambda_{H^{1}}^{r}$, we fix $h=\frac{1}{64}$, $\Delta t=10^{-3}$, $\delta=6.25\times10^{-2}$, $N=5$, and vary $r$. With these choices, $h^{2 m}=\mathcal{O}\left(10^{-8}\right)$ and $\left\|S^{r}\right\|_{2}=\mathcal{O}\left(10^{2}\right)-\mathcal{O}\left(10^{5}\right)$. Thus, the theoretical ADL-ROM approximation error bound predicts the following rate of convergence of $E_{L^{2}}^{ADL-ROM}$ with respect to $\Lambda_{H^{1}}^{r}$:

\begin{equation}
E_{L^{2}}^{ADL-ROM}=\mathcal{O}\left(\Lambda_{H^{1}}^{r}\right).
\label{eqn:theoretical-roc-lambda-h1}
\end{equation}

The ADL-ROM approximation error $E_{L^{2}}^{ADL-ROM}$ is listed in Table \ref{tbl:r_Scaling_ROM} for increasing $r$ values. The corresponding linear regression, which is shown in Fig.~\ref{fig:r_Scaling_ROM}, indicates the following ADL-ROM approximation error rate of convergence with respect to $\Lambda_{H^{1}}^{r}$:

\begin{equation*}
E_{L^{2}}^{ADL-ROM}=\mathcal{O}\left(\left(\Lambda_{H^{1}}^{r}\right)^{1.50}\right).
\end{equation*}
Thus, the theoretical rate of convergence is numerically recovered.
We note that the rate of convergence displayed by the numerical results in Fig. \ref{fig:r_Scaling_ROM} (i.e., $1.50$) is higher than the theoretical rate of convergence in~\eqref{eqn:theoretical-roc-lambda-h1}.
We believe that this higher rate of convergence in Fig. \ref{fig:r_Scaling_ROM} is in part due to the outlier data point in the top right corner of the plot, which corresponds to the largest value of $\Lambda_{H^{1}}^{r}$.
This $r$ value yields the lowest ROM resolution results, which are less accurate than the results corresponding to the other $r$ values.

\section{Conclusions}
    \label{sec:conclusions}

The impetus for the approximate deconvolution method is similar to that of LES, where we are interested in finding accurate solutions for large structures in under-resolved, convection-dominated fluid flow simulations at a minimal computational cost. In a ROM context, we cannot expect to fully resolve the small flow scales, but these small scales usually have important implications for solution accuracy. 
This has previously been addressed with Leray type models to smooth the non-linearities just enough \citep{gunzburger2020leray,xie2018numerical}. Unfortunately, Leray models can be overly diffusive and smear the solution. The approximate deconvolution Leray method for ROMs,
recently introduced by the authors in~\citet{sanfilippo2023approximate}, offers a compromise candidate between the classical Galerkin projection and the potentially over-smoothing Leray filter method. 

This paper presented the first numerical analysis of the ROM approximate deconvolution operator (Section \ref{sec:rom_ad}). 
In particular, we proved an AD approximation bound in Lemma \ref{lemma:Layton_213_Rom}, and we illustrated numerically the corresponding rates of convergence in Section \ref{sec:AD_ROC}. 
We also performed the first numerical analysis of the new ADL-ROM (Section~\ref{sec:adl-rom-numerical-analysis}).
Specifically, we proved an {\emph{a priori}} error bound for the ADL-ROM in Theorem \ref{thm:32}, 
and numerically illustrated the theoretical results in Section \ref{sec:ADL-ROM_ROC}. 

The rates of convergence demonstrated in Section \ref{sec:numerical-results} confirmed our theoretical results, but more importantly provided an idea of what to expect when altering the filter radius $\delta$ in both general AD scenarios and in our application to the NSE. 
Indeed, when working on realistic problems, difficult decisions need to be made about accuracy and resolved scales vs.\ computational expense. 
Without a mathematical framework for how these filters behave, tuning can only be done on an ad hoc basis.
Thus, the resulting ROMs lack robustness and need to be recalibrated for every new computational setting.
To our knowledge, the important role of the filter radius on the ROM results has been investigated only numerically, see, e.g., the ROM-adaptive lengthscale in~\citet{mou2023energy} and the sensitivity studies on $\delta$ in~\citet{girfoglio2021pod,tsai2023time}.
The current study provides a numerical analysis strategy for investigating the effect of the filter radius on the ROM results, and represents a rigorous alternative to the aforementioned approaches.

This paper and the preceding paper \citep{sanfilippo2023approximate} have integrated classical AD methods with ROM regularization methods. A clear avenue for future research, then, is to integrate ROMs with more modern AD procedures. The field of image processing, which inspires many of these techniques, is changing rapidly with the advent of accessible machine learning, but these advances have not propagated into ROM research. There is considerable space to investigate the performance of new AD methods in a ROM context, both those reminiscent of classical methods \citep{Benning2018inverse} and those dependent on a machine learning architecture \citep{Ongie2020deeplearning}.
Furthermore, while ROM regularization can be a cheap and effective modeling technique, ROM closure methods \citep{xie2017approximate} provide an alternative framework to solving the problem of under-resolved scales and can also benefit from advances in ROM AD methods. 
Finally, providing numerical analysis support and, in particular, proving {\emph{a priori}} error bounds for the new ROM AD strategies is an important research avenue for creating reliable ROMs.

\section*{Acknowledgments}
We acknowledge the European Union's Horizon 2020 research and innovation program under the Marie Sk\l{}odowska-Curie Actions, grant agreement 872442 (ARIA).
FB also acknowledges the PRIN 2022 PNRR project ``ROMEU: Reduced Order Models for Environmental and Urban flows'' funded by the European Union -- NextGenerationEU under the National Recovery and Resilience Plan (NRRP), Mission 4 Component 2, CUP J53D23015960001. 
AS and FB further acknowledge INdAM GNCS.
TI and IM acknowledge support through National Science Foundation grants DMS-2012253 and CDS\&E-MSS-1953113.

We thank William Layton for clarifying the proof of inequality (2.22) in ~\citet{layton2008numerical}, which helped motivate proving two different bounds in Lemma~\ref{thm:ROM211_split1} (see Remark~\ref{remark:filter-stability-bounds}).

\makeatletter
\bibliographystyle{plain}

\makeatother
\bibliography{traian}

\end{document}